\documentclass[10pt]{article}
\usepackage{tikz}
\usepackage[english]{babel}
\usepackage{hyperref}
\usepackage{amsmath,amsfonts,amssymb,amsthm,enumerate,mathabx}
\usepackage{latexsym}
\usepackage{makeidx}
\usepackage{amscd}
\usepackage{hyperref}
\usepackage{vmargin}
\usepackage{cite}
\setmarginsrb{23mm}{16mm}{23mm}{16mm}%
            {14mm}{8mm}{14mm}{14mm}

\numberwithin{equation}{section}

\newtheorem{definition}{Definition}[section]
\newtheorem{lemma}[definition]{Lemma}
\newtheorem{theorem}[definition]{Theorem}
\newtheorem{corollary}[definition]{Corollary}

\newtheorem{proposition}[definition]{Proposition}

\newtheorem{em-example}[definition]{Example}
\newtheorem{em-def}[definition]{Definition}        
\newtheorem{em-remark}[definition]{Remark}         
\newtheorem{em-question}[definition]{Question}
\newtheorem{question}[definition]{Question}

\newenvironment{example}{\begin{em-example} \em }{ \end{em-example}}
\newenvironment{remark}{\begin{em-remark} \em }{\end{em-remark}}

\newcommand{\N}{\mathbb N}
\newcommand{\C}{\mathfrak C}

\newcommand{\Q}{\mathbb Q}
\newcommand{\Z}{\mathbb Z}

\newcommand{\T}{\mathcal{T}}
\newcommand{\f}{\phi}

\def\mod{\mathbf{Mod}}

\renewcommand{\hom}{\mathrm{\hom}}
\def\mod{\mathbf{Mod}}
\def\Hom{\mathrm{Hom}}

\def\Im{\mathrm{Im}}

\global\def\mod#1{\mathrm{Mod}\text{-}#1}
\global\def\lmod#1{#1\text{-}\mathrm{Mod}}

\def\coker{\mathrm{CoKer}}

\def\F{\mathcal F}
\renewcommand\ker{\mathrm{Ker}}

\def\M{\mathbb M}

\def\I{\mathcal I}
\def\E{\mathcal E}

\def\A{\mathfrak A}
\def\D{\mathfrak D}

\def\F{\mathcal{F}}
\renewcommand\hom{\mathrm{Hom}}
\def\Gdim{\mathrm{G.dim}}
\def\coker{\mathrm{CoKer}}

\def\Ch{{\bf Ch}}
\def\A{\mathcal A}

\def\Spec{\mathrm{Spec}}

\def\id{id}

\def\Sp{\mathrm{Sp}}

\def\link{\rightsquigarrow}

\def\2link{\stackrel{{}_2}{\link}}

\def\Ann{\mathrm{Ann}}

\def\Im{\mathrm{Im}}

\global\def\mod#1{\mathrm{Mod}\text{-}#1}
\def\F{\mathcal F}

\def\W{\mathcal W}

\def\L{\mathcal L}

\def\M{\mathfrak M}
\def\loc{{\bf L}}
\def\Q{{\bf Q}}
\def\S{{\bf S}}
\def\tor{{\bf T}}

\def\l{{\bf L}}
\def\t{{\bf T}}
\def\tors{\mathrm{Tors}}
\def\Inj{\mathrm{Inj}}
\def\cone{\mathrm{cone}}
\def\Tow{\mathrm{Tow}}

\input xy
\xyoption{all}

\title{Model approximations for relative homological algebra}

\author{Simone Virili}

\begin{document}

\maketitle
\abstract{Recently, Chach\'olski, Neeman, Pitsch, and Scherer studied, in a series of three papers, model approximations for the unbounded category of cochain complexes over a commutative ring. These approximations allow to construct relative injective resolutions with respect to particular choices of injectives. In this paper we define similar model approximations for cochain complexes on general Grothendieck categories generalizing the previous constructions and reaching a better understanding of the whole picture.}

\smallskip\noindent
--------------------------------

 


\section{Introduction} 
{ Model categories} were introduced in the late sixties by Quillen \cite{Quillen}. A {\em model category} $(\M,\W,\cal B,\cal C )$ is a bicomplete category $\M$ with three distinguished classes of morphisms, called respectively weak equivalences, fibrations and cofibrations, satisfying some axioms (see Definition \ref{model_cat}). An important property of model categories is that one can invert weak equivalences, obtaining a new category, called the {\em homotopy category}. 

\smallskip
Given a Grothendieck category $\C$ and denoting by $\Ch(\C)$ the category of unbounded  cochain complexes on $\C$, there always exists a model structure on $\Ch(\C)$, called {\em injective model structure}, whose class of weak equivalences is given by the quasi-isomorphisms (see Example \ref{inj_mod_struc}). The homotopy category arising this way is the unbounded derived category ${\bf D}(\C)$ (see \cite{hovey} for more details).  Having a model structure on $\Ch(\C)$ allows for concrete constructions of derived functors.

\smallskip
The concept of { model approximation} was introduced by Chach{\'o}lski and Scherer \cite{PS} with the aim of constructing homotopy limits and colimits in arbitrary model categories. The advantage of model approximations is that it is easier in general to prove that a given category has a  model approximation than defining a model structure on it. On the other hand, model approximations allow to construct derived functors and to define the homotopy category. More precisely, consider a category $\C$ with a distinguished class of morphisms $\W_\C$ and a model category $(\M,\W,\cal B,\cal C )$. A {\em model approximation} of $(\C,\W_\C)$  by $(\M,\W,\cal B,\cal C )$ is a pair of adjoint functors 
$$l :  \xymatrix{\C\ar@<2pt>[r]\ar@<-2pt>@{<-}[r]&\M}  : r$$
such that $l$ sends the elements of $\W_\C$ to elements of $\W$ and some other technical conditions (for more details see Definition \ref{mod_appr}).

\smallskip
Chach\'olski,  Pitsch, and Scherer \cite{CPS1} introduced a useful model approximation for the category of unbounded complexes $\Ch(R)$ over a ring $R$, whose homotopy category is the usual derived category ${\bf D}(R)$. This construction encodes in a pair of adjoint functors the classical ideas to construct injective resolutions of unbounded complexes. The aim of the successive  paper \cite{CPS2} of the same three authors together with Neeman,  is to modify the construction of \cite{CPS1} in order to obtain a ``model approximation for relative homological algebra". 

\smallskip
Let us be more specific about the meaning of {\em relative homological algebra} in this context. Consider a Grothendieck category $\C$, roughly speaking, an {\em injective class} is a suitable class $\I$ of objects of $\C$ that is meant to represent a ``different choice" for the injective objects in the category (see Definition \ref{def.inj.class}). Given  a class $\I$ one says that a morphism $\phi$ in $\C$ is an {\em $\I$-monomorphism} provided $\Hom_\C(\phi, I)$ is an injective map for all $I\in\I$. In particular, $\I$-monomorphisms are the usual monomorphisms provided $\I$ is the class of all injective objects. 
Chach\'olski,  Pitsch, and Scherer \cite{CPS} studied the injective classes of the category of modules $\mod R$ over a commutative ring $R$, classifying all the injective classes of injective objects  in terms of suitable families of ideals, called {\em saturated sets}.  In case $R$ is also Noetherian, their classification reduces to a bijection between the set of injective classes of injectives and the family of generalization closed subsets of the spectrum of $R$.

\smallskip
Given a Grothendieck category $\C$ and an injective class of injective objects $\I$, one says that a morphism of unbounded complexes $\phi^\bullet:M^\bullet\to N^\bullet$ is an {\em $\I$-quasi-isomorphism} provided $\hom_{\C}(\phi^\bullet,I)$ is a quasi-isomorphism of complexes of Abelian groups for all $I\in \I$. The following questions naturally arise:

\begin{question}\label{quest_iniziale}
In the above notation, denote by $\W_\I$ the class of $\I$-quasi-isomorphisms, then
\begin{enumerate}[\rm (1)]
\item is it possible to find a model approximation for $(\Ch(\C),\W_\I)$? If so, what does the homotopy category of such approximation look like?
\item Is it possible to find, in analogy with \cite{CPS1}, a model approximation via towers of model categories of half bounded complexes (see Subsection \ref{towers}) which allows to construct inductively relative injective resolutions of unbounded complexes?
\end{enumerate}
\end{question}
Chach\'olski, Neeman, Pitsch, and Scherer \cite{CPS2} partially answer the above questions in case the category $\C$ is the category of modules over a commutative Noetherian ring $R$. In fact, they show that, if $R$ has finite Krull dimension, then it is possible to find an approximation via towers of models whose homotopy category is obtained inverting $\I$-quasi-isomorphisms, for any injective class of injectives $\I$. On the other hand, if the Krull dimension of $R$ is not finite, there always exists an injective class of injectives $\I$ for which this is not possible.

\smallskip
In this paper, we try to tackle the above questions in the general setting of Grothendieck categories. Recall that, given a Grothendieck category $\C$, a {\em torsion theory} $\tau=(\T,\F)$ is a pair of subclasses of $\C$, maximal with respect to the condition that $\Hom_\C(T,F)=0$ for all $T\in\T$ and $F\in\F$. Furthermore, $\tau$ is {\em hereditary} provided $\F$ is closed under taking injective envelopes (or, equivalently, $\T$ is closed under taking sub-objects). The key observation from which this paper starts is that there is a bijective correspondece between injective classes of injectives and hereditary torsion theories induced by the following correspondence:
$$\tau=(\T,\F)\xymatrix{{}\ar@{|->}[r]&{}} \I_\tau=\{\text{injective objects in $\F$}\}\,.$$
Then one can see that a morphism is an $\I_{\tau}$-monomorphism if an only if its kernel is $\tau$-torsion (i.e., it belongs to $\T$). Recall that it is possible to associate to any hereditary torsion theory $\tau$ a localization of the category $\C$, which is encoded in the following pair of adjoint functors:
\begin{equation*}\xymatrix{
\C
\ar@<2,5pt>[rr]^{\Q_\tau}& &
\C/\T\ar@<2,5pt>[ll]^{\S_\tau}\, ,}\end{equation*}
where $\C/\T$ is a Grothendieck category, which is called the {\em localization of $\C$ at $\tau$}. The {\em $\tau$-quotient functor} $\Q_\tau$ is exact, while the {\em $\tau$-section functor} $\S_\tau$ is just left exact in general (see Section \ref{tt_sec} for details). With this formalism one can see that a morphism $\phi$ in $\C$ is an $\I_\tau$-monomorphism if and only if $\Q_\tau(\phi)$ is a monomorphism in $\C/\T$.

\smallskip
The bijection with hereditary torsion theories allows us to prove that, in a given Grothendieck category $\C$, there is just a set (as opposed to a proper class) of injective classes of injective objects. Furthermore, if $\C$ has Gabriel dimension (see Subsection \ref{gab_cat_subs}) and it is stable (which just means that any hereditary torsion class in $\C$ is closed under taking injective envelopes), then the injective classes of injective objects are in bijection with the family of generalization closed subsets of the Gabriel spectrum of $\C$ (see Corollary \ref{coro_inj}). In particular, our classifications are natural generalizations of the classifications described in \cite{CPS}, see Corollaries \ref{sat_sat} and \ref{sat_noeth}. 

\smallskip
Finally, we can answer part (1) of Question \ref{quest_iniziale} in full generality. First of all, one extends the functors $\Q_\tau$ and $\S_\tau$ to the categories $\Ch(\C)$ and $\Ch(\C/\T)$ (just applying them degree-wise), this gives rise to an adjunction. Abusing notation, we use the same symbols for these new functors. Then, one proves that a morphism of complexes $\phi^\bullet$ in $\Ch(\C)$ is an $\I_\tau$-quasi-isomorphism if and only if $\Q_\tau(\phi^\bullet)$ is a quasi-isomorphism in $\Ch(\C/\T)$. Furthermore, if we endow $\Ch(\C/\T)$ with the canonical injective model structure, there is a model approximation
$$\Q_\tau :  \xymatrix{(\Ch(\C),\W_{\I_\tau})\ar@<2pt>[r]\ar@<-2pt>@{<-}[r]&\Ch(\C/\T)}  : \S_\tau\,.$$
The homotopy category associated with this model approximation has a very concrete form: it is precisely the derived category ${\bf D}(\C/\T)$, see Theorem \ref{main_general}. 

\smallskip
The answer to part (2) of Question \ref{quest_iniziale} is more delicate. First of all, one needs to understand what fails in the construction of \cite{CPS2} and why. Indeed, Chach\'olski et al. try to directly approximate the category $(\Ch(\C),\W_{\I_\tau})$ by a category of ``towers of models". We try a slightly different approach, which looks for a model approximation in two successive steps. In this way we are able to identify what goes wrong in the approximation process: the quotient category $\C/\T$ may fail to be $(Ab.4^*)$-$k$ for all $k\in\N$ (see Definition \ref{quasi_exact_prod}), even if $\C$ is a very nice category (say a category of modules over a commutative Noetherian ring). This, a fortiori trivial, observation is sufficient to explain why the towers of models cannot approximate $(\Ch(\C),\W_{\I_\tau})$ in general. In fact, there is no reason for an object $X^\bullet$ in the homotopy category $\bf D(\C/\T)$ of $(\Ch(\C),\W_{\I_\tau})$ for being isomorphic to the limit of its truncations if $\C/\T$ is not $(Ab.4^*)$-$k$ for any $k\in\N$. Now, going back to the original question, we can partially answer as follows: one can always find a model approximation of $(\Ch(\C),\W_{\I_\tau})$ by towers of model categories of half-bounded complexes over $\C/\T$ provided $\C/\T$ is $(Ab.4^*)$-$k$ for some $k\in\N$, see Theorem \ref{main_th} and Corollary \ref{main_co}.

\subsection*{Structure of the paper} 
In Section \ref{tt_sec} we give some standard results about localization of Grothendieck categories also introducing the  notion of relative Gabriel dimension of a category. A large part of the section is devoted to the study of the derived functors of the torsion functors, called {\em local cohomologies}. We give detailed proofs for all the results about local cohomology, without claiming they are original, because we were not able to find a unique reference where the results we need are proved in the required generality. As an application of this generalized local cohomology theory, we  give explicit conditions to ensure that a quotient category $\C/\T$ of a Grothendieck category $\C$ over a hereditary torsion theory is ($Ab.4^*$)-$k$ for some $k\in\N$. \\
In Section \ref{Sec_inj_class}, we describe the bijection between hereditary torsion theories and injective classes of injectives, also presenting the announced classifications and giving some examples.\\
Finally, in Section \ref{Sec_mod_approx} we give the answers to Question \ref{quest_iniziale} that we briefly exposed above.

\subsection*{Acknowledgement}
I would like to thank my
Ph.D. advisor Dolors Herbera for her encouragement and for several discussions on the subjects of this paper. I also want to thank Wolfgang Pitsch for sending me copies of his preprints and for discussing with me many of his ideas.

\subsection*{Notations and terminology}
Let $\C$ be a category. Given a subclass $\A$ of objects of $\C$, we define the following two classes
$$\A^{\perp}=\{X\in \C:\hom_\C(A,X)=0\,, \; \forall A\in \A\}\ \ \text{ and }\ \ {}^{\perp}\A=\{X\in \C:\hom_\C(X,A)=0\,, \; \forall A\in \A\}\,,$$
which are called respectively {\em right} and {\em left orthogonal class} to $\A$.
\\
Let  $(I, \leq)$ be a directed set. Recall that a  directed system in $\C$ is a family of objects and maps $(X_i,\phi_{j,i}:X_i\to X_j)_{i\leq j\in I}$. If no confusion is possible, we denote such system by $(X_i,\phi_{j,i})_{I}$. If the direct limit of $(X_i,\phi_{j,i})_{I}$ exists, the limit is denoted by $\varinjlim_IX_i$. We follow analogous conventions for an inverse systems and inverse limits. Furthermore, given a set $J$ and a family of objects $\{X_j:j\in J\}$ we denote by $\prod_JX_j$ and $\bigoplus_JX_j$ the product and the coproduct respectively. If $\alpha$ is a cardinal and $X$ is an object, we let $X^\alpha$ and $X^{(\alpha)}$ be the product and the coproduct respectively of $\alpha$-many copies of $X$.

\smallskip
Let now $\C$ be a {\em Grothendieck category}. This means exactly that $\C$ is a cocomple abelian category with enough injectives, a generator and exact coproducts. 
\\
Given an object $X$ in $\C$, we write $Y\leq X$ (resp., $Y<X$) to mean that $Y$ is a sub-object (resp., a proper sub-object) of $X$. Recall that the family of all sub-objects of $X$ has a lattice structure, we denote this lattice by $\L(X)$. Given $Y_1$ and $Y_2\in \L(X)$, we let $Y_1+Y_2$ and $Y_1\cap Y_2$ be respectively the least upper bound and the greatest lower bound of $Y_1$ and $Y_2$ in $\L(X)$. 
Given a directed system  $(X_i,\phi_{j,i})_{I}$ of sub-objects of $X$, their direct limit $\varinjlim_IX_i$ corresponds with the least upper bound of the family $\{X_i:i\in I\}$ in $\L(X)$; we denote this particular direct limit also by $\sum_{i\in I}X_i$.
\\
We denote by $E(X)$ the injective envelop of $X$. An object $C\in \C$ is {\em uniform} if and only if $E(C)$ is indecomposable or, equivalently, if given two non-trivial sub-objects $A,$ $B\leq C$ we have that $A\cap B\neq 0$.

A cochain complex $X^\bullet$ is a sequence of objects $\{X^n:n\in\Z\}$ and differentials $d^n:X^n\to X^{n+1}$ for all $n\in\Z$, such that $d^{n+1}d^n=0$ for all $n\in\Z$. A morphism of complexes $\phi^\bullet:X^\bullet\to Y^\bullet$ is a family of morphisms $\{\phi^n\colon X^n\to Y^n:n\in\Z\}$, such that $d^{n+1}\phi^n=\phi^{n+1}d^n$ for all $n\in\Z$.\\
We denote by $\Ch(\C)$ the category of cochain complexes and of morphisms of complexes. Furthermore, for all $n\in\Z$, we let $\Ch^{\geq n}(\C)$ be the full subcategory of $\Ch(\C)$ of all cochain complexes $X^\bullet$ such that $X^{k}=0$ for all $k<n$. \\
Given a cochain complex $X^\bullet$, we denote by $X^{\geq n}\in\Ch^{\geq n}(\C)$ the {\em $n$-th truncation} of $X^\bullet$, that is
$$X^{\geq n}: \xymatrix{ &\cdots\ar[r]&0\ar[r]&0\ar[r]&\coker(d^{n-1})\ar[r]&X^{n+1}\ar[r]&X^{n+2}\ar[r]&\cdots}\,.$$
Given two complexes $X^\bullet$ and $Y^\bullet\in\Ch(\C)$, $\mathcal{H}{om}(X^\bullet,Y^\bullet)$ is a cochain complex of abelian groups whose $n$-th component is
$$\mathcal{H}{om}(X^\bullet,Y^\bullet)^n=\prod_{i+j=n}\hom_\C(X^i,Y^j)\,.$$
Furthermore, given a morphism $\phi^\bullet:X^\bullet\to Y^\bullet$, the {\em mapping cone} $\mathrm{cone}(\phi^\bullet)$ is a cochain complex whose $n$-th component is $X^{n+1}\oplus Y^n$ and whose differentials are represented in matrix form as follows:
$$\begin{pmatrix}
d&0\\
\phi^n&d
\end{pmatrix}:X^{n+1}\oplus Y^n\xymatrix{\ar[r]&} X^{n+1}\oplus Y^n\,.$$
Finally, given a ring $R$, we denote by $\lmod R$ the category of left $R$-modules over $R$.
\section{Torsion theories and localization}\label{tt_sec}

\subsection{Torsion theories}\label{tt_subs}
Let $\C$ be a Grothendieck category and let $\A\subseteq \C$ be a subclass. Recall that $\A$ is a Serre class if, given a short exact sequence
$$0\to A\to B\to C\to 0\,,$$
$B$ belongs to $\A$ if and only if both $B$ and $C$ belong to $\A$.  Furthermore, $\A$ is a {\em hereditary torsion class} (or {\em localizing class}) if it is Serre and it is closed under taking arbitrary direct sums. On the other hand,  $\A$ is a {\em hereditary torsion free class} provided it is closed under taking sub-objects, extensions, products and injective envelopes. 

\begin{definition}
A {\em hereditary torsion theory} $\tau$ in $\C$ is a pair of classes $(\T, \F)$ such that 
\begin{enumerate}[\rm --]
\item the class $\T$ of {\em $\tau$-torsion} objects is a hereditary torsion class;
\item the class $\F$ of {\em $\tau$-torsion free} objects is a hereditary torsion free class;
\item $(\T)^{\perp}=\F$ and $^{\perp}(\F)=\T$.
\end{enumerate}
In this paper the symbols $\tau$, $\T$ and $\F$ are always used to denote a torsion theory, a torsion class and a torsion free class respectively. Since all the torsion theories in the sequel are hereditary, we just say ``torsion theory'', ``torsion class" and ``torsion free class" to mean respectively ``hereditary torsion theory'', ``hereditary torsion class'' and ``hereditary torsion free class''. 
\end{definition}
Of course, there is some redundant information in the above definition. 

\begin{example}
\begin{enumerate}[\rm (1)]
\item The pairs $(0, \C)$ and $(\C,0)$ are torsion theories. We call them respectively the {\em trivial} and the {\em improper} torsion theory. 
\item Given an injective object $E$ in $\C$, one can define a torsion theory $\tau=(\T,\F)$, with $\T={}^{\perp}\{E\}$ and $\F=\T^{\perp}$; such $\tau$ is said to be the {\em torsion theory cogenerated by $E$}.
\end{enumerate}
\end{example}

\begin{lemma}\label{cog}
Let $\C$ be a Grothendieck category and let $\tau=(\T,\F)$ be a torsion theory in $\C$. Then $\tau$ is cogenerated by an injective object $E$.
\end{lemma}
\begin{proof}[Sketch of proof] 
Take a generator $G$ of $\C$ and let $E$ be the product of all the injective envelopes of the $\tau$-torsion free quotients of $G$. Then, $\T$ coincides with ${}^{\perp}\{E\}$ and $\F=\T^\perp$ is the class of all the objects that embed in some product of copies of $E$.
\end{proof}

The sketched proof of Lemma \ref{cog} also shows that the torsion theories in a Grothendieck category $\C$ form a set, not a proper class (in fact, one can bound the cardinality of this set by the cardinality of the power set of the family of quotients of a chosen generator $G$ of $\C$). 

\begin{definition}
We denote by $\tors(\C)$ the poset of all the torsion theories on $\C$, ordered as follows: given $\tau=(\T,\F)$ and $\tau'=(\T',\F')\in\tors(\C)$, 
$$\tau'\preceq \tau \ \text{ if and only if }\ \T'\subseteq \T\ \text{ if and only if }\ \F\subseteq \F' \,.$$  
When $\tau'\preceq \tau$, we say that $\tau$ is a {\em generalization} of $\tau'$, while $\tau'$ is a {\em specialization} of $\tau$.
\end{definition}

\subsubsection{Stable torsion and locally Noetherian Grothendieck categories}
The notion of stable torsion theory was introduced by Gabriel in \cite{Gabriel}, see also \cite{papp} and \cite{stab}.

\begin{definition} 
A torsion theory $\tau=(\T,\F)\in \tors(\C)$ is {\em stable} if $\T$ is closed under taking injective envelopes. Furthermore, $\C$ is {\em stable} if any $\tau\in \tors(\C)$ is stable. 
\end{definition}

\begin{lemma}\label{split_inj_lemma}
Let $\C$ be a Grothendieck category, let $\tau\in\tors(\C)$ and let $E$ be an injective object in $\C$. If $\tau$ is stable, then
\begin{equation*}\label{split_inj}E\cong \t_\tau(E)\oplus E/\t_\tau(E)\,.\end{equation*} 
\end{lemma}
\begin{proof} 
Identify $\t_\tau(E)$ and its injective envelope with sub-objects of $E$, so that $\t_\tau(E)\leq E(\t_\tau(E))\leq E$. Since $E(\t_\tau(E))$ is $\tau$-torsion by stability, $E(\t_\tau(E))\leq \t_\tau(E)$. Thus, $\t_\tau(E)=E(\t_\tau(E))$. Having proved that $\t_\tau(E)$ is injective, the direct sum decomposition claimed in \eqref{split_inj} follows. 
\end{proof}

A Grothendieck category $\C$ is {\em locally Noetherian} if it has a set of Noetherian generators, equivalently, any object in $\C$ is the direct union of the directed family of its Noetherian sub-objects. In the following proposition and corollary we collect some results about locally Noetherian categories:

\begin{proposition}\label{prop_noeth}{\rm\cite{Gabriel, Ste}}
Let $\C$ be a locally Noetherian Grothendieck category. The following statements hold true. 
\begin{enumerate}[\rm (1)]
\item Direct limits of injective objects are injective.
\item Given an injective object $E$, there exist a set $I$, a family of pairwise non-isomorphic indecomposable injective objects $\{E_i:i\in I\}$ and a family of non-trivial cardinals $\{\alpha_i:i\in I\}$ such that
$$E\cong \bigoplus_{i\in I}E_i^{(\alpha_i)}\,.$$
Furthermore, given a set $I'$, a family of pairwise non-isomorphic indecomposable injective objects $\{E'_i:i\in I'\}$ and a family of non-trivial cardinals $\{\alpha'_i:i\in I'\}$ such that $E\cong \bigoplus_{i\in I'}(E'_i)^{(\alpha'_i)}$, there is a bijective map $\sigma:I\to I'$ such that $E_i\cong E_{\sigma(i)}'$ and $\alpha_i=\alpha_{\sigma(i)}'$.
\end{enumerate}
\end{proposition}

\begin{corollary}\label{coro_noeth}
Let $\C$ be a locally Noetherian Grothendieck category and let $E=\bigoplus_{i\in I}E_i^{(\alpha_i)}$ be an injective object written as a coproduct of indecomposable injectives (for all $i\in I$, $E_i$ is indecomposable injective and $\alpha_i$ is a non-trivial ordinal). Then, 
$$\T=\bigcap\left\{\T_i:i\in I\right\}\,,$$
where $\T={}^{\perp}\{E\}$ and $\T_i={}^{\perp}\{E_i\}$ for all $i\in I$.
\end{corollary}

\subsection{Localization}
The concept of torsion theory is intimately related to that of localization.
\begin{definition}
A {\em localization} of $\C$ is a pair of adjoint functors $\Q:\C\rightleftarrows \D :\S$, where $\S$ is fully faithful and $\Q$ is exact. In this situation, $\D$ is called a {\em quotient category}, $\Q$ is a {\em quotient functor} and $\S$ is a {\em section functor}. The composition $\loc=\S\circ\Q:\C\to \C$ is called the {\em localization functor}.
\end{definition}

\begin{lemma}\label{Krause_loc}{\rm \cite{holoc}}
Let $\C$ be a Grothendieck category, let $(\Q,\S)$ be a localization of $\C$ and denote by $\loc=\S\circ\Q:\C\to \C$ the localization functor. For all $X\in \C$ there is a natural isomorphism $\loc(X)\cong \loc(\loc(X))$. 
\end{lemma}

One can encounter slightly different definitions of localization in other contexts, see for example \cite{Localization}.
We can now explain the connection between localizations and torsion theories. Indeed, starting with a localization $\Q:\C\rightleftarrows \D:\S$ and letting $\loc=\S\circ\Q$, 
$$\ker(\loc)=\{X\in \C : \loc(X)=0\}=\{X\in \C : \Q(X)=0\}=\ker (\Q)$$ 
is a torsion class (use the exactness of $\Q$ and the fact that it is a left adjoint). Hence, the localization $(\Q,\S)$  induces a torsion theory $(\ker(\Q),\ker(\Q)^{\perp})$.

\medskip
On the other hand, one can construct a localization out of a torsion theory. 
\begin{definition}
Let $\C$ be a Grotheindieck category and $\tau=(\T,\F)\in \tors(\C)$. An object $X\in \C$ is {\em $\tau$-local} if $X\in\F$ and $E(X)/X\in \F$. 
The full subcategory of $\C$ of all the $\tau$-local objects is denoted by $\C/\T$. 
\end{definition}

The definition of the localization induced by a torsion theory depends on the following lemma.

\begin{lemma}{\rm \cite{Gabriel}}
Let $\C$ be a Grotheindieck category and let $\tau=(\T,\F)\in \tors(\C)$. Then, the canonical inclusion $\C/\T\to \C$ has a left adjoint functor which is exact.
\end{lemma}

\begin{definition}
Let $\C$ be a Grotheindieck category and let $\tau=(\T,\F)\in \tors(\C)$. The inclusion  \mbox{$\S_\tau:\C/\T\to \C$} is called {\em $\tau$-section functor}, while its left adjoint functor $\Q_\tau:\C\to \C/\T$ is called {\em $\tau$-quotient functor}. The composition $\loc_\tau=\S_\tau\Q_\tau$ is called {\em $\tau$-localization functor}.\\
We say that $\tau$ is {\em exact} if $\loc_\tau$ (or, equivalently, $\S_\tau$) is an exact functor.
\end{definition}

It is possible to give an explicit description of the action of $\loc_\tau$ on objects. For that, we need to introduce another functor:
\begin{definition}
Let $\C$ be a Grotheindieck category and let $\tau=(\T,\F)\in \tors(\C)$. The {\em $\tau$-torsion functor} $\tor_\tau:\C\to \T$ is defined as follows:
\begin{enumerate}[\rm --]
\item $\tor_\tau(X)=\sum\,\{T\leq X:T\in \T\}$, for any object $X\in \C$;
\item $\tor_\tau(\phi):\tor_\tau(X)\to \tor_\tau(Y)$ is the restriction restriction of $\phi:X\to Y$, for any morphism $\phi$ in $\C$. 
\end{enumerate}
\end{definition}
It is an exercise to show that $\tor_\tau$ is a right adjoint to the inclusion $\T\to \C$. 

\begin{proposition}\label{loc_esp}
Let $\C$ be a Grothendieck category, let $\tau=(\T,\F)\in \tors(\C)$ and let $X,$ $Y\in \C$. The following statements hold true:
\begin{enumerate}[\rm (1)]
\item $\Q_\tau(T)=0$ and $\loc_\tau(T)=0$ for all $T\in \T$;
\item $\loc_\tau(X)\cong \loc_\tau(X/\tor_\tau(X))$;
\item if there is an exact sequence $0\to X\to L\to T\to 0$ with $L$ $\tau$-local and $T\in\T$, then $\loc_\tau(X)\cong L$; 
\item letting $X'=X/\tor_\tau(X)$, 
$$\loc_\tau(X)\cong\pi^{-1}(\tor_\tau(E(X')/X'))\,,$$
where $\pi:E(X')\to E(X')/X'$ is the canonical projection. 
\item $\hom_{\C/\C_{\tau}}(\Q_\tau (X),\Q_\tau (Y))\cong\hom_{\C}(X,Y)$, provided both $X$ and $Y$ are $\tau$-local.
\end{enumerate}
\end{proposition}
\begin{proof}
(1) Suppose, looking for a contradiction, that $\Q_\tau(T)\neq0$. We get $\hom_{\C/\T}(\Q_\tau(T),\Q_\tau(T))\neq 0$ but then, by the adjointeness of $\Q_\tau$ and $\S_\tau$, $\hom_{\C}(T,\loc_\tau(T))\neq 0$ and, by definition, $\loc_\tau(T)$ is $\tau$-torsion free and not trivial, which is a contradiction.\\
\noindent (2) By part (1) and the exactness of $\Q_\tau$, if we apply $\Q_\tau$ to the exact sequence $0\to \tor_\tau(X)\to X\to X/\tor_\tau(X)\to 0$, we get $\Q_\tau(X)\cong \Q_\tau(X/\tor_\tau(X))$. Now apply $\S_\tau$ to this isomorphism.\\
\noindent (3) Apply $\loc_\tau$ to the exact sequence $0\to X\to L\to T\to 0$. By part (1) and left exactness, $\loc_\tau(L)\cong \loc_\tau(X)$. Now apply Lemma \ref{Krause_loc}.\\
\noindent (4) By part (2), $\loc_\tau(X)\cong \loc_\tau(X')$. Furthermore, it follows  from the definition that $\pi^{-1}(\tor_\tau(E(X')/X'))$ is $\tau$-local, so, by part (3), $\loc_\tau(X)\cong\pi^{-1}(\tor_\tau(E(X')/X'))$.\\
\noindent (5) By adjointness,  $\hom_{\C/\C_{\tau}}(\Q_\tau (X),\Q_\tau (Y))\cong\hom_{\C}( X,\loc_\tau (Y))\cong \hom_{\C}(X,Y)$.
\end{proof}

%

\begin{corollary}\label{prod}
Let $\C$ be a Grothendieck category and let $\tau=(\T,\F)\in\tors(\C)$ be a torsion theory. Consider a family $\{X_i:i\in I\}\subseteq \C/\T$, then
\begin{equation*}\prod_{i\in I}X_i\cong \Q_{\tau}\left(\prod_{i\in I}\S_\tau (X_i)\right)\, .\end{equation*}
\end{corollary}
\begin{proof} 
Being right adjoint, $\S_\tau$ commutes with limits, thus $\S_\tau\left(\prod_{i\in I}X_i\right)\cong\prod_{i\in I}\S_\tau (X_i)$. Apply $\Q_{\tau}$ to conclude.
\end{proof}

\subsection{Gabriel categories and Gabriel spectrum}\label{gab_cat_subs}
Let $\C$ be a Grothendieck category. The {\em Gabriel filtration} of $\C$ is a transfinite chain 
$$\{0\}=\C_{-1}\subseteq \C_0\subseteq \dots \subseteq \C_\alpha\subseteq \dots$$
of torsion classes defined as follows:
\begin{enumerate}[--]
\item $\C_{-1}=\{0\}$;
\item suppose that $\alpha$ is an ordinal for which $\C_{\alpha}$ has already been defined, an object $C\in \C$ is said to be {\em $\alpha$-cocritical} if  $C'\notin \C_\alpha$ and $C/C'\in\C_\alpha$, for any non-trivial sub-object $C'\leq C$. We let $\C_{\alpha+1}$ be the smallest torsion class containing $\C_{{\alpha}}$ and all the ${\alpha}$-cocritical objects;
\item if $\lambda$ is a limit ordinal, we let $\C_\lambda$ be the smallest hereditary torsion class containing $\bigcup_{\alpha<\lambda}\C_\alpha$.
\end{enumerate}
\begin{definition}
A Grothendieck category $\C$ is said to be a {\em Gabriel category} if $\C=\bigcup_{\alpha}\C_\alpha$.
\end{definition}
Let $\C$ be a Gabriel category and let $\tau=(\T,\F)$ be a torsion theory. One can show that $\C/\T$ is a Gabriel category as well (showing by induction that $\Q_\tau(\C_\alpha) \subseteq (\C/\T)_{\alpha}$ for all $\alpha$).

\smallskip
For any ordinal $\alpha$, we let $\tau_\alpha=(\C_\alpha, \C_\alpha^{\perp})$; in what follows, we write $\alpha$-torsion (resp., torsion free, local, ...) instead of $\tau_\alpha$-torsion (resp., torsion free, local, ...). Furthermore, we let $\tor_\alpha:\C\to \C_{\alpha}$, $\S_\alpha:\C/\C_{\alpha}\to \C$ and $\Q_\alpha:\C\to\C/\C_{\alpha}$ be respectively the $\alpha$-torsion, the $\alpha$-section and the $\alpha$-quotient functors. Abusing notation, we use the same symbols for the functors $\tor_\alpha:\C_{\alpha+1}\to \C_{\alpha}$ and $\Q_\alpha:\C_{\alpha+1}\to\C_{\alpha+1}/\C_{\alpha}$, induced by restriction.
\begin{definition}
An object $C$ in $\C$ is {\em cocritical} if it is $\alpha$-cocritical for some $\alpha$. The torsion theories that can be cogenerated by the injective envelope of a cocritical object are said to be {\em prime}.\\ 
In this paper, the symbol $\pi$ is always used to denote a prime torsion theory.
\end{definition}

The following lemma is folklore.

\begin{lemma}
Let $\C$ be a Gabriel category, let $C\in \C$ be an object and let $\alpha$ be an ordinal. The following are equivalent:
\begin{enumerate}[\rm (1)]
\item $C$ is $\alpha$-cocritical;
\item $C$ is $\alpha$-torsion free and $\Q_\alpha(C)$ is simple;
\item there exists a simple object $S\in \C/\C_{\alpha}$ such that $C$ embeds in $\S_\alpha(S)$.
\end{enumerate}
\end{lemma}

The next lemma is well-know but we were not able to find this exact statement in the literature. Recall that a Grothendieck category $\mathfrak D$ is said to be semi-Artinian if $\Gdim(\mathfrak D)=0$, equivalently,  every object in $\mathfrak D$ has a simple sub-object. 

\begin{lemma}\label{indec_prime}
Let $\C$ be a Gabriel category and let $E$ and $E'\in \C$ be injective objects; the following statements hold true:
\begin{enumerate}[\rm (1)]
\item $E$ is indecomposable if and only if there exists a cocritical object $C$ such that $E\cong E(C)$;
\item if $E$ and $E'$ are indecomposables and cogenerate the same torsion theory, then $E\cong E'$.
\end{enumerate}
\end{lemma}
\begin{proof}[Sketch of the proof]
(1) Suppose $E=E(C)$, where $C$ is $\alpha$-cocritical for some ordinal $\alpha$. If $\alpha=-1$ then $C$ is simple and so $E=E(C)$ is indecomposable.  If $\alpha>0$, proceed in the category $\C/\C_\alpha$ as in the case when $\alpha=-1$ to show that $E$ is indecomposable.

On the other hand, suppose $E$ is indecomposable and let $\alpha$ be the smallest ordinal such that $\tor_{\alpha+1}(E)\neq 0$. If $\alpha=-1$, then $E$ has a simple socle $S$, so $E\cong E(S)$. If $\alpha>-1$, then apply the same argument to $\Q_\alpha(E)$ and deduce that $E=E(\S_\alpha(C))$ for some simple object $C$ in $\C/\C_\alpha$.

\noindent
(2) Let $\tau=(\T,\F)$ and $\tau'=(\T',\F')$ be the torsion theories cogenerated by $E$ and $E'$ respectively and suppose $\tau=\tau'$. By part (1), $E=E(C)$ for some cocritical object $C$, furthermore $C\leq E\in \F=\F'$, thus $\hom_\C(C,E')\neq 0$. So, let $\phi:C\to E'$ be a non-trivial map and notice that it is necessarily a monomorphism. Thus, $E'\cong E(\phi(C))\cong E(C)=E$. 
\end{proof}

\begin{definition}
Given a Gabriel category $\C$, the {\em $\alpha$-Gabriel spectrum} $\Sp^{\alpha}(\C)$ is the family of isomorphism classes of injective envelopes of $\alpha$-cocritical objects. The {\em Gabriel spectrum} $\Sp(\C)$ of $\C$ is the family of isomorphism classes of indecomposable injective objects in $\C$. \\
Given a subset $S\subseteq\Sp(\C)$, we say that $S$ is {\em generalization closed} (resp., {\em specialization closed}) if it contains all the prime torsion theories that are generalizations (specializations) of its members. 
\end{definition}
By Lemma \ref{indec_prime}, in a Gabriel category $\C$,
$$\Sp(\C)=\bigcup_{\alpha}\Sp^\alpha(\C)\,.$$ 
Furthermore, $\Sp(\C)$ corresponds bijectively to the family of prime torsion theories. We identify these two sets and we write $\pi\in\Sp(\C)$ (or $\pi\in \Sp^\alpha(\C)$) to mean that $\pi$ is a prime torsion theory. Furthermore, we let $E(\pi)$ be a representative of the isomorphism class of indecomposable injectives which cogenerate $\pi$.

\begin{lemma}\label{stab_ort}
Let $\C$ be a Gabriel category, let $\pi,\pi'\in\Sp(\C)$ and consider the following conditions:
\begin{enumerate}[\rm (1)]
\item $\pi\preceq \pi'$;
\item $\hom_\C(E(\pi'),E(\pi))\neq 0$.
\end{enumerate}
Then, (1) implies (2). If $\pi$ is stable, also the converse holds.
\end{lemma}
\begin{proof}
By definition, $\pi\preceq \pi'$ if and only if any $\pi'$-torsion free object is $\pi$-torsion free. In this case, $E(\pi')$ is $\pi$-torsion free, thus $\hom_\C(E(\pi'),E(\pi))\neq 0$. On the other hand, if $\hom_\C(E(\pi'),E(\pi))\neq 0$ and $\pi$ is stable, then $E(\pi')$ is not $\pi$-torsion, thus it is $\tau$-torsion free (see Lemma \ref{split_inj_lemma}) and so $\pi\preceq \pi'$.
\end{proof}

\begin{theorem}\label{class_gabriel}
Let $\C$ be a Grothendieck category and let $\tau=(\T,\F)$ be a torsion theory. Define the following subsets of $\Sp(\C)$:
\begin{enumerate}[\rm --]
\item $S(\tau)=\{\pi\in \Sp(\C): \tor_\tau(E(\pi))\neq 0\}$;
\item $G(\tau)=\{\pi\in \Sp(\C): \tor_\tau(E(\pi))=0\}$.
\end{enumerate}
Then, $S(\tau)\cup G(\tau)=\Sp(\C)$ and this is a disjoint union. Furthermore, given $\tau'\in\tors(\C)$, $\tau=\tau'$ if and only if $G(\tau)=G(\tau')$ if and only if $S(\tau)=S(\tau')$.
\\
If $\C$ is stable, then $S(\tau)$ and $G(\tau)$ are respectively specialization and generalization closed. Furthermore, any specialization (resp., generalization) closed subset of $\Sp(\C)$ is of the form $S(\tau)$ (resp., $G(\tau)$) for some $\tau\in \tors(\C)$ and $S(-)$ (resp., $G(-)$) induces a bijection between $\tors(\C)$ and the set of specialization (resp., generalization) closed subsets of $\Sp(\C)$.  
\end{theorem}
\begin{proof}
The first part of the statement comes from {\rm \cite[Corollaries 1 and 2 on page 384]{Gabriel}}. Assume that $\C$ is stable and let $\pi\preceq \pi'\in\Sp(\C)$, that is, $\hom_\C(E(\pi'),E(\pi))\neq 0$ (see Lemma \ref{stab_ort}). If $\tor_\tau(E(\pi'))\neq 0$, then $E(\pi')\in \T$ (by stability), thus any proper quotient of $E(\pi')$ is $\tau$-torsion. Hence, $\tor_\tau(E(\pi))\neq 0$. We proved that $S(\tau)$ is specialization closed, the fact that $G(\tau)$ is generalization closed follows from the fact that it is the complement of $S(\tau)$.
\\
Finally, let $G$ be a generalization closed subset of $\Sp(\C)$ and let 
$$\underline\T={}^{\perp}\{E(\pi):\pi\in G\}\,,\ \underline\F=\underline\T^{\perp}\ \text{ and }\ \underline\tau=(\underline\T,\underline\F)\,.$$
Then $(\Sp(\C)\setminus G)\cup G=\Sp(\C)=S(\underline\tau)\cup G(\underline\tau)$ and it is easy to see that $G\subseteq G(\underline\tau)$. Let $\pi'\in \Sp(\C)\setminus G$. If, looking for a contradiction, $E(\pi')\notin\underline \T$, then there exists $\pi\in G$ such that $\hom_\C(E(\pi'),E(\pi))\neq 0$. By Lemma \ref{stab_ort} $\pi'$ is a generalization of $\pi$ and so $\pi'\in G$, which is a contradiction. Hence, $\Sp(\C)\setminus G\subseteq S(\underline\tau)$ and so  $S(\underline\tau)=\Sp(\C)\setminus G$ and $G(\underline\tau)=G$.
\end{proof}

\subsection{(Relative) Gabriel dimension}
The concept of Gabriel dimension was introduced in \cite{Gabriel} (under the name of ``Krull dimension") and systematically studied in \cite{GR} and in many other papers and books after that. We introduce here a relative version of this invariant.

\begin{definition}
Let $\C$ be a Gabriel category, let $\tau=(\T,\F)$ be a torsion theory and let $X\in \C$ be an object. We define respectively the {\em $\tau$-Gabriel dimension} of $\C$ and the {\em $\tau$-Gabriel dimension} of $X$ as follows
$$\Gdim_\tau(\C)=\min\{\alpha:\C/\T= (\C/\T)_\alpha\}\ \text{ and }\ \Gdim_\tau(X)=\min\{\alpha:\Q_\tau(X)\in (\C/\T)_\alpha\}\,.$$
When $\tau=(0,\C)$ is the trivial torsion theory, the $\tau$-Gabriel dimension is called {\em Gabriel dimension} and we denote it respectively by $\Gdim(\C)$ and $\Gdim(X)$.
\end{definition}

\begin{remark}\label{dim_inj}
Let $\C$ be a stable Gabriel category and let $E$ be an indecomposable injective object. By Lemma \ref{indec_prime}, there is an ordinal $\alpha$ and an $\alpha$-cocritical object $C$ such that $E\cong E(C)$. By construction $\Gdim(C)=\alpha+1$ and, by stability, $\Gdim(E(C))=\Gdim(C)$. This shows that the Gabriel dimension of an indecomposable injective object is always a successor ordinal.
\end{remark}

In the following lemma we collect some useful properties of Gabriel dimension; their proof is given in \cite{Virili_tesi} for $\tau=(0,\C)$, the general case is completely analogous. 

\begin{lemma}\label{pre1}{\rm\cite[Lemma 2.13 and Corollaries 2.15-2.16]{Virili_tesi}}
Let $\C$ be a Gabriel category and let $\tau=(\T,\F)$ be a torsion theory. Then:
\begin{enumerate}[\rm (1)]
\item $\Gdim_\tau(\C)=\sup\{\Gdim_\tau(X): X\in\C\}$;
\item if $Y\leq X\in \C$, then $\Gdim_\tau(X)=\max\{\Gdim_\tau(Y),\Gdim_\tau(X/Y)\}$;
\item if $\{X_i:i\in I\}$ is a family of objects in $\C$, then $\Gdim_\tau(\bigoplus_I X_i)=\sup_I\Gdim_\tau(X_i)$;
\item let $\alpha$ be an ordinal and $X \in \C$, then $X\in \C_{\alpha+1}$ if and only if there exists an ordinal $\sigma$ and a continuous chain 
$0=Y_0< Y_1< \dots< Y_\sigma=X$, such that $Y_{i+1}/Y_i$ is either $\alpha$-cocritical or $\alpha$-torsion for every $i< \sigma$;
\item if $N\in\C$ is a Noetherian object, then $\Q_\tau(N)$ is Noetherian and $\Gdim_\tau(N)$ is a successor ordinal. Furthermore, there exists a finite series $0=Y_0< Y_1<\dots< Y_k=N$ such that $Y_{i}/Y_{i-1}$ is cocritical for all $i=1,\dots,k$;
\item let $\lambda$ be a limit ordinal and $X \in \C$, then $X\in \C_{\lambda}$ if and only if $X=\bigcup_{\alpha<\lambda}\tor_\alpha(X)$;
\item $\C_{\alpha+1}/\C_\alpha$ is semi-Artinian for all $\alpha<\Gdim(\C)$.
\end{enumerate}
\end{lemma}

Given a torsion theory  $\tau=(\T,\F)\in\tors(\C)$, let $\C'=\C/\T$ and consider a torsion theory $\tau'=(\T',\F')\in\tors(\C')$. The following class of objects of $\C$ is a torsion class:
$$\T_{\tau\circ\tau'}=\{X\in \C:\Q_\tau(X)\in \T'\}\,.$$
\begin{definition}
We denote by $\tau\circ\tau'$ the torsion theory whose torsion class is $\T_{\tau\circ\tau'}$.
\end{definition}
Notice that, just by definition, the quotient functors relative to $\tau$, $\tau'$ and $\tau\circ\tau'$ fit in the following commutative diagram:
$$\xymatrix{\C\ar@/_20pt/[rrrr]_{\Q_{\tau\circ\tau'}}\ar[rr]^{\Q_\tau}&&\C/\T\ar[rr]^(0.35){\Q_{\tau'}}&&\C/\T_{\tau\circ\tau'}\cong \C'/\T'}\,.$$

\begin{lemma}\label{rel_circ_gab}
Let $\C$ be a Gabriel category, let $\tau=(\T,\F)\in\tors(\C)$, let $\C'=\C/\T$, denote by $\tau_\alpha$ the torsion theory in $\C'$ whose torsion class is $(\C')_\alpha$ (the $\alpha$-th member of the Gabriel filtration of $\C'$), and let $X\in\C$. Then:
\begin{enumerate}[\rm (1)]
\item $\Gdim_\tau(X)=\alpha+1$ if and only if $\Gdim_{\tau\circ\tau_\alpha}(X)=0$;
\item $\Gdim_{\tau\circ\tau_\alpha}(X)=-1$ implies that $\Gdim_{\tau}(X)\leq \alpha$.
\end{enumerate}
\end{lemma}
\begin{proof}
(1) $\Gdim_\tau(X)=\Gdim(\Q_\tau(X))=\alpha+1$ if and only if $\Q_\tau(X)\in (\C')_{\alpha+1}\setminus (\C')_\alpha$, that is, $\Gdim_{\tau\circ\tau_\alpha}(X)=\Gdim( \Q_{\tau_\alpha}(\Q_\tau(X)))=0$.

\smallskip\noindent
(2)  $\Gdim_{\tau\circ\tau_\alpha}(X)=-1$ if and only if $\Q_{\tau\circ\tau_\alpha}(X)=0$, that is, $\Q_\tau(X)\in\ker(\Q_{\tau_\alpha})=(\C')_\alpha$. Equivalently, $\Gdim_\tau(X)=\Gdim(\Q_\tau(X))\leq\alpha$.
\end{proof}

\begin{lemma}\label{TTK<Gdim}
Let $\C$ be a stable Gabriel category, let $\tau\in\tors(\C)$ and let $\pi =(\T,\F)$ and $\pi'=(\T',\F')$ be two distinct prime torsion theories.\\
If $\Gdim_\tau(E(\pi))=\Gdim_\tau(E(\pi'))>-1$, then $\hom_\C(E(\pi),E(\pi'))=0$. 
\end{lemma}
\begin{proof}
By Remark \ref{dim_inj}, $\Gdim_\tau(E(\pi))=\Gdim_\tau(E(\pi'))=\beta+1$ for some ordinal $\beta\geq -1$. Denote by $\tau_\beta\in\tors(\C/\T)$ the torsion theory whose torsion class is $(\C/\T)_\beta$. Then, $\Gdim_{\tau\circ\tau_\beta}(E(\pi))=\Gdim_{\tau\circ\tau_\beta}(E(\pi'))=0$ and, by stability, both $E(\pi)$ and $E(\pi')$ are $\tau\circ\tau_\beta$-torsion free, so $\tau\circ\tau_\beta$-local. By Proposition \ref{loc_esp},
$$\hom_{\C}(E(\pi),E(\pi'))\cong\hom_{\C/\T_{\tau\circ\tau_\beta}}(E(\pi),E(\pi'))\,,$$
so there is no loss in generality if we assume that $\tau$ is the trivial torsion theory and $\beta=-1$. Suppose now that $\hom_\C(E(\pi),E(\pi'))\neq0$, that is, $E(\pi)$ is not $\pi'$-torsion, thus, by stability, $E(\pi)$ is $\pi'$-torsion free. Now, since we are assuming $\Gdim(E(\pi))=0$, there exists a simple object $S\in\C$ such that $E(\pi)\cong E(S)$ and $\hom_\C(S,E(\pi'))\neq 0$. By the simplicity of $S$ and the uniformity of $E(\pi')$ we obtain that $E(\pi')\cong E(S)\cong E(\pi)$, which is a contradiction.
\end{proof}

\begin{corollary}\label{coroTTK}
Let $\C$ be a stable Gabriel category, let $\tau=(\T,\F)\in\tors(\C)$ and let $\pi \neq \pi'\in\Sp(\C)$.\\
If $E(\pi),E(\pi')\notin \T$ and $\hom_\C(E(\pi),E(\pi'))\neq 0$, then $\Gdim_\tau(E(\pi))>\Gdim_\tau(E(\pi'))$. 
\end{corollary}
\begin{proof}
Suppose, looking for a contradiction that $\Gdim_\tau(E(\pi))\leq \Gdim_\tau(E(\pi'))$, let $\Gdim_\tau(E(\pi))=\alpha+1$ and let
\begin{enumerate}[\rm --]
\item $\C_1=\C/\T$, $\bar\tau_1=(\T_1,\F_1)\in\tors(\C_1)$, where $\T_1=(\C_1)_\alpha$, $\tau_1=\tau\circ\bar\tau_1\in\tors(\C)$, $E_1=\Q_{\tau_1}(E(\pi))$  and $E_1'=\Q_{\tau_1}(E(\pi'))$;
\item $\C_2=\C_1/\T_1$, $\bar\tau_2=(\T_2,\F_2)\in\tors(\C_2)$, where $\T_2={}^{\perp}\{E_1'\}$ and $\tau_2=\tau\circ\bar\tau_2\in\tors(\C)$;
\item $\C_3=\C_2/\T_2$.
\end{enumerate}
Both $E(\pi)$ and $E(\pi')$ are $\tau_1$-local and so, by Proposition \ref{loc_esp},
$$\hom_{\C_2}(E_1,E_1')\cong\hom_\C(E(\pi),E(\pi'))\neq 0\,.$$
This means that $E_1$ is not $\bar \tau_2$-torsion, thus both $\Gdim_{\tau_2}(E(\pi))$ and $\Gdim_{\tau_2}(E(\pi'))$ are strictly bigger than $-1$. On the other hand, $\Gdim_{\tau_2}(E(\pi))\leq\Gdim_{\tau_1}(E(\pi))=0$ (see Lemma \ref{rel_circ_gab}), while $\Gdim_{\tau_2}(E(\pi'))=0$ (since, given a cocritical sub-object $C$ of $E(\pi')$, $\loc_{\tau_2}(C)$ is simple). This is a contradiction by Lemma \ref{TTK<Gdim}.
\end{proof}

\subsection{Local cohomology}\label{loc_coh_subs}
In this subsection we  introduce a very general notion of local cohomology. The definitions and many arguments in the proofs are adapted directly from existing papers like \cite{Albu-Nastasescu1}, \cite{Albu-Nastasescu2}, \cite{G}, \cite{Golan-libro}, \cite{Golan-Raynaud} and many others. We give here complete proofs for making this paper as self-contained as possible and because, to the best of the author's knowledge, there is no book or paper with a comprehensive exposition of these matters in a setting as general as we need in the sequel.

\begin{definition}
Let $\C$ be a Grothendieck category and let $\tau=(\T,\F)\in\tors(\C)$. The {\em $n$-th $\tau$-local cohomology} $\Gamma_\tau^n:\C\to \T$ is the $n$-th right derived functor of $\tor_\tau$. 
\end{definition}
It is difficult to study the properties of local cohomology in full generality, so we need to impose one or more of the following hypotheses on the ambient category $\C$ in almost all of our results: 
\begin{enumerate}[\rm ({Hyp.}1)]
\item $\C$ is stable.
\item $\C$ is locally Noetherian.
\item all the prime torsion theories on $\C$ are exact.
\end{enumerate}

\begin{example}
Let $\C=\lmod R$ be the category of left $R$-modules over a ring $R$. When $\C$ is stable $R$ is said to be a stable ring, while $\C$ is locally Noetherian precisely when $R$ is left Noetherian. When $\C$ satisfies (Hyp.1), (Hyp.2) and (Hyp.3), $R$ is said to be {\em left effective}. Examples of left effective rings include:
\begin{enumerate}[\rm (1)]
\item commutative Noetherian rings;
\item left Noetherian Azumaya algebras, see \cite[page 173]{Golan-Raynaud};
\item prime hereditary Noetherian quasi-local rings which are bounded orders in their classical rings of fractions, see \cite[Example 2.3]{Golan-Raynaud}.
\end{enumerate}
\end{example}

\begin{lemma}\label{coho1}
Let $\C$ be a Grothendieck category, let $\tau=(\T,\F)\in\tors(\C)$ be stable and let $X\in\C$. Then,
\begin{enumerate}[\rm (1)]
\item $\Gamma_\tau^n(X)=0$ for all $n>0$, provided $X$ is $\tau$-torsion;
\item there is a natural isomorphism ${ \Gamma}_\tau^{n}(X)\cong { \Gamma}_\tau^{n}(X/\t_\tau(X))$, for all $n>0$;
\item there is a natural isomorphism $\Gamma_{\tau}^{n+1}(X)\cong\Gamma_\tau^{n}(E(X)/X)$, for all $n>0$;
\item if $X$ is $\tau$-torsion free, then $\Gamma_\tau^1(X)\cong \tor_\tau(E(X)/X)$.
\item there is a natural isomorphism ${ \Gamma}_\tau^{n+1}(X)\cong{\rm R} \l^{n}_\tau(X)$, for all $n>0$;
\item if $M$ is $\tau$-torsion free, then $\Gamma_\tau^1(X)\cong \l_\tau(X)/X$.
\end{enumerate}
\end{lemma}
\begin{proof}
(1) Let $E^\bullet(X)$ be a minimal injective resolution of $M$, so $E^0(X)=E(X)$ and $E^n(X)=E(d(E^{n-1}(X)))$ for all $n>0$. Using the fact that $\tau$ is stable, we obtain that $E^n(X)$ is $\tau$-torsion for all $n\in\N$ and so $E^\bullet(X)=\t_\tau(E^\bullet (X))$ is an exact complex in all degrees but, eventually, in the $0$-th degree.

\smallskip\noindent
(2) Consider the short exact sequence $0\to \t_\tau(X)\to X\to X/\t_\tau(X)\to 0$. 
This gives a long exact sequence in cohomology 
\begin{align*}0\to\Gamma_\tau^0(\t_\tau(X)) \to \Gamma_\tau^0(X)&\to \Gamma_\tau^0(X/\t_\tau(X))\to\\
& \to \Gamma_\tau^1(\t_\tau(X))\to \Gamma_\tau^1(X)\to \Gamma_\tau^1(X/\t_\tau(X))\to \Gamma_\tau^2(\t_\tau(X))\to\cdots\,,\end{align*}
which implies the desired isomorphism as, by part (1), we have that $\Gamma_\tau^n(\t_\tau(X))=0$, for all $n>0$.

\smallskip\noindent
(3)--(4) Consider the short exact sequence $0\to X\to E(X)\to E(X)/X\to 0$. This gives a long exact sequence in cohomology
\begin{align*}0\to \t_\tau(X)&\to \t_{\tau}(E(X))\to \t_\tau(E(X)/X)\to \Gamma_\tau^1(X)\to \Gamma_\tau^1(E(X))\to\\ 
&\Gamma_\tau^1(E(X)/X)\to \Gamma_\tau^2(X)\to \Gamma_\tau^2(E(X))\to \Gamma_\tau^2(E(X)/X)\to \Gamma_\tau^3(X)\to \Gamma_\tau^3(E(X))\to\dots\end{align*}
which implies the isomorphism in (3) as, $\Gamma_\tau^n(E(X))=0$ for all $n>0$, being $E(X)$ is injective. If $X\in \F$, then $E(X)\in \F$ by stability and so  $\t_{\tau}(E(X))=0$, proving (4).

\smallskip\noindent
(5)--(6) Fix  an injective resolution $0\to X\to E^\bullet$. By Lemma \ref{split_inj_lemma}, $E^n\cong \t_\tau(E^n)\oplus E^n/\t_\tau(E^n)$ for all $n\in\N$. Consider the complex $\t_\tau(E^\bullet)$ and the quotient complex  $E^\bullet/ \t_\tau(E^\bullet)$, which are both complexes of injective objects. Notice that there is a short exact sequence in $\Ch(\C)$
\begin{equation*}\label{split_inj_complex}0\to \t_\tau(E^\bullet)\to E^\bullet\to E^\bullet/ \t_\tau(E^\bullet)\to 0\,.\end{equation*}
The cohomologies of the complex $\t_\tau(E^\bullet)$ are exactly the $\tau$-local cohomologies of $M$, while the cohomologies of $E^\bullet$ are all trivial but, eventually, the $0$-th cohomology. Furthermore, $E^n/\t_\tau(E^n)$ is $\tau$-torsion free and injective, so it is $\tau$-local; in particular,  $\l_\tau(E^n)\cong E^n/\t_\tau(E^n)$ for all $n\in\N$. We obtain an isomorphism of complexes $E^\bullet/\t_\tau(E^\bullet)\cong \l_\tau(E^\bullet)$, which shows that the cohomologies of the complex $E^\bullet/\t_\tau(E^\bullet)$ give exactly the right derived functors of the localization functor $\l$. Thus, we have a long exact sequence
\begin{align*}
0\to \t_\tau(X) \to X&\to \l_\tau(X)\to \Gamma_\tau^1(X)\to 0\to {\rm R}^1\l_\tau(X)\to\\
&\to \Gamma_\tau^2(X)\to 0\to {\rm R}^2\l_\tau(X)\to \Gamma_\tau^3(X)\to 0\to {\rm R}^3\l_\tau(X)\to\cdots\,,\end{align*}
which gives the desired isomorphisms.
\end{proof}

An immediate consequence of parts (5) and (6) of the above proposition is the following
\begin{corollary}\label{vanish_ex}
Let $\C$ be a Grothendieck category, let $\tau\in\tors(\C)$ be stable and exact, and let $X\in\C$. Then,
$$\Gamma_\tau^n(X)=0 \ \ \forall n>1\,.$$
If $X$ is $\tau$-local, then also $\Gamma_\tau^0(X)=\Gamma_\tau^1(X)=0$.
\end{corollary}

\begin{corollary}\label{1leq2vanish}
Let $\C$ be a Grothendieck category and let $\tau_2\preceq \tau_1\in\tors(\C)$ be stable. Then, 
$$(\Gamma_{\tau_1}^i(X)=0\,, \ \forall 0\leq i\leq n)\ \Rightarrow\ (\Gamma_{\tau_2}^i(X)=0\,, \ \forall 0\leq i\leq n)$$
for all $X\in \C$ and $n\in \N$.
\end{corollary}
\begin{proof}
We prove our statement by induction on $n\in \N$.
\\
If $n=0$, the result is clear as $0=\Gamma_{\tau_1}^0(X)\cong\t_{\tau_1}(X)\supseteq \t_{\tau_2}(X)\cong\Gamma_{\tau_2}^0(X)$.
\\
If $n\geq1$ and suppose the result holds for any smaller integer. Consider the following isomorphisms:
\begin{enumerate}[\rm (a)]
\item $\Gamma_{\tau_1}^0(X)\cong\Gamma_{\tau_2}^0(X)=0$, as in the case $n=0$, in particular $X$ is both $\tau_1$ and $\tau_2$-torsion free;
\item $\Gamma_{\tau}^{1}(X)\cong\Gamma_\tau^{0}(E(X)/X)$ for $\tau\in\{\tau_1,\tau_2\}$, by (a) and Lemma \ref{coho1}(4).
\item $\Gamma_{\tau}^{k}(X)\cong\Gamma_\tau^{k-1}(E(X)/X)$ for all $1<k\leq n$ and $\tau\in\{\tau_1,\tau_2\}$, by Lemma \ref{coho1}(3).
\end{enumerate}
By (b) and (c), we have that $\Gamma_{\tau_1}^{i}(E(X)/X)=0$ for all $0\leq i\leq n-1$ and so, by inductive hypothesis, $\Gamma_{\tau_{2}}^{i}(E(X)/X)=0$ for all $0\leq i\leq n-1$. Applying again (b) and (c), $\Gamma_{\tau_2}^{i}(X)=0$ for all $1\leq i\leq n$. Hence, adding (a), $\Gamma_{\tau_2}^{i}(X)=0$ for all $0\leq i\leq n$ which is what we wanted to prove.
\end{proof}


\begin{proposition}\label{comm_dir}
Let $\C$ be a Grothendieck category satisfying (Hyp.2) and let $\tau\in\tors(\C)$. Then, 
\begin{enumerate}[\rm (1)]
\item given two objects $X$ and $M\in \C$, we have a natural isomorphism 
$$\hom_\C(X,\t_\tau M)\cong \varinjlim_Y \hom_\C(X/Y,M)\,,$$ with $Y$ ranging in the family of sub-objects of $X$ such that $X/Y\in\T$ (ordered by reverse inclusion);
\item all the $\tau$-local cohomology functors commute with direct limits. 
\end{enumerate}
\end{proposition}
\begin{proof}
(1) Let $Y$ be a sub-object of $X$ such that $X/Y\in\T$. For any morphism $\phi:X/Y\to M$,  $\phi(X/Y)\leq \t_\tau(M)\leq M$ and so
$\hom_\C(X/Y,\t_\tau(M))\cong\hom_\C(X/Y,M)$. Furthermore, there is an injective map 
\begin{equation}\label{inj_homo_induced}-\circ p_Y:\hom_\C(X/Y,\t_\tau(M))(\cong \hom_\C(X/Y,M))\longrightarrow \hom_\C(X,\t_\tau(M))\end{equation}
where $p_Y:X\to X/Y$ is the canonical projection. By the universal property of the direct limit, there is a unique map $\Phi$, making the following diagrams commutative, whenever $Y_1\leq Y_2\leq X$ and $X/Y_1\in\T$:
$$\xymatrix{\hom_\C(X/Y_1,M)\ar@/_-15pt/[rrrd]\ar@/_-15pt/[rd]|{-\circ p_{Y_1}}\\
&\hom_\C(X,\t_\tau(M))&&\varinjlim_Y\hom_\C(X/Y,M)\, .\ar[ll]|{\exists!\, \Phi}\\
\hom_\C(X/Y_2,M)\ar[uu]\ar@/_15pt/[ru]|{-\circ p_{Y_2}}\ar@/_15pt/[rrru]}$$
Now, $\Phi$ is injective by the injectivity of the maps described in \eqref{inj_homo_induced} and the commutativity of the above diagram; furthermore, one can show that $\Phi$ is surjective as follows: an element $\phi\in \hom_\C(X,\t_\tau(M))$ belongs to the image of $\Phi$ if and only if there exists $Y\leq X$ such that $X/Y\in \T$ and there is a morphism $\psi:X/Y\to M$ such that $\phi=\psi p_Y$. Given $\phi\in \hom_\C(X,\t_\tau(M))$, we easily get that $\phi(X)\in \T$ and so, letting $Y=\ker(\phi)$, we have that $X/Y\in \T$. Furthermore,  there is an induced (mono)morphism $\psi:X/Y\to \t_\tau(X)$ such that $\phi=\psi p_Y$, as desired.

\smallskip\noindent
(2) According to \cite[Proposition 3.6.2]{Grothendieck}, it suffices to verify that
\begin{enumerate}[\rm (a)]
\item ${\rm R}^0\t_\tau\cong\t_\tau$ commutes with direct limits;
\item $\varinjlim_\Lambda E_\alpha$ is $\t_\tau$-acyclic (i.e., $\Gamma_\tau^n(\varinjlim_\Lambda E_\alpha)=0$ for all $n>0$) for any directed system $(E_\alpha,\phi_{\beta,\alpha})_{ \Lambda}$ of injective objects. 
\end{enumerate}
Let $(M_\alpha,\phi_{\beta,\alpha}:M_\alpha\to M_\beta)_{\alpha\leq \beta\in \Lambda}$ be a directed system in $\C$, over a directed set $\Lambda$, denote by $M$ the direct limit $\varinjlim_{\Lambda}M_\alpha$ and denote by $\phi_\alpha:M_{\alpha}\to M$ the canonical maps of the direct limit. We have to verify that 
$\t_\tau( M)=\varinjlim_\Lambda\t_\tau(M_{\alpha})$. It is not difficult to see that $\t_\tau(M)\geq \varinjlim_\Lambda\t_\tau(M_{\alpha})$ (since direct limit of torsion objects is torsion). For the other inclusion we have to show that, for any Noetherian suboject $N\leq \t_\tau(M)$, there is $\beta\in\Lambda$ such that $N\leq \phi_{\beta}(\t_\tau(M_\beta))$. Indeed, since $M=\bigcup_{\Lambda}\phi_\alpha(M)$, there exists $\alpha$ such that $N\leq \phi_\alpha(M)$, thus there exists $N'\leq M_\alpha$ such that $\phi_\alpha(N')=N$. Furthermore, since $N$ is torsion, $N'/(N'\cap \ker(\phi_\alpha))$ is torsion as well and, since $\ker(\phi_\alpha)=\bigcup_{\beta \geq \alpha}\ker(\phi_{\beta,\alpha})$ and $N'\cap \ker(\phi_\alpha$ is Noetherian, there exists $\beta$ such that $N'\cap \ker(\phi_\alpha)=N'\cap\ker(\phi_{\beta,\alpha})$. By these computation,  $N''=\phi_{\beta,\alpha}(N')\leq \t_\tau(M_\beta)$ such that $N=\phi_\beta(N'')\leq \phi_\beta(\t_\tau(M_\beta))$.\\
Part (b) follows by Proposition \ref{prop_noeth} (1), and the fact that injective objects are $F$-acyclic for any left exact functor $F$.
\end{proof}

%

\begin{lemma}\label{Gdimloc<}
Let $\C$ be a Grothendieck category satisfying (Hyp.1) and (Hyp.2), let $\bar\pi\in\Sp(\C)$, let $\tau=(\T,\F)\in\tors(\C)$ and let $C$ be a cocritical object such that $E(C)\cong E(\bar\pi)$.\\ If $-1<\Gdim_\tau(C)<\infty$, then $\Gdim_\tau(\l_{\bar\pi}(C)/C)<\Gdim_\tau(C)$.
\end{lemma}
\begin{proof}
By hypothesis $\Gdim_\tau(C)>-1$. Furthermore, if $\Gdim_\tau(C)=n+1$ for some $n\in \N$, we can denote by $\tau_n\in\tors(\C/\T)$ the torsion theory whose torsion class is $(\C/\T)_n$. Then, $\Gdim_{\tau\circ\tau_n}(C)=0$. Thus, there is no loss of generality in assuming that $\Gdim_{\tau}(C)=0$ (otherwise substitute $\tau$ by $\tau\circ\tau_n$ and then use part (2) of Lemma \ref{rel_circ_gab}). Now, if $\Gdim_\tau(C)=0$, we have to show that $\l_{\bar\pi}(C)/C\in \T$. By Proposition \ref{prop_noeth}, there exist a set $I$, a family of prime torsion theories $\{\pi_i=(\T_i,\F_i):i\in I\}\subseteq \Sp(\C)$ and a family of non-trivial cardinals $\{\alpha_i:i\in I\}$ such that $E\cong \bigoplus_{i\in I}E(\pi_i)^{(\alpha_i)}$. In view of Corollary \ref{coro_noeth}, $\T=\bigcap_{i\in I}\T_i$ and so we reduced to prove that 
\begin{equation}\label{claim)()(}\hom_\C(\l_{\bar\pi}(C)/C,E(\pi_i))=0\,,\end{equation}
for all $i\in I$. Let $i\in I$, if $\pi_i=\bar \pi$, then $\l_{\bar \pi}(C)/C\cong \t_{\bar\pi}(E(\bar \pi)/C)$ is $\pi_i$-torsion by construction, so \eqref{claim)()(} follows. On the other hand, if $\pi_i\neq\bar \pi$, suppose looking for a contradiction, that $\hom_\C(\l_{\bar \pi}(C)/C,E(\pi_i))\neq 0$ which implies $\hom_\C(\l_{\bar \pi}(C),E(\pi_i))\neq 0$ which, by the injectivity of $E(\pi_i)$, implies $\hom_\C(E(\bar \pi),E(\pi_i))\neq 0$. By Corollary \ref{coroTTK}, $0=\Gdim_\tau(E(\bar\pi))>\Gdim_\tau(E(\pi_i))$, equivalently, $E(\pi_i)\in \T$, that is \linebreak $\hom_\C(E(\pi_i),E)=0$, which is clearly a contradiction. 
\end{proof}

The following theorem is an improved version of \cite[Proposition 2.4]{Golan-Raynaud}.

\begin{theorem}\label{vanish_loc_co_hom}
Let $\C$ be a Grothendieck category satisfying (Hyp.1), (Hyp.2) and (Hyp.3), let $\tau\in\tors(\C)$ and let $X\in\C$. Then, $\Gamma_\tau^n(X)\neq 0$ implies $\Gdim_\tau(X)+2\geq n$. 
\end{theorem}
\begin{proof}
By (Hyp.3), $X$ is the direct union of its Noetherian sub-objects. By Proposition \ref{comm_dir}, the vanishing of $\tau$-local cohomologies on Noetherian objects implies their vanishing on $X$. Thus we can suppose $X$ to be Noetherian. By Lemma \ref{pre1}(5), there exist sub-objects $0=Y_0\leq Y_1\leq \dots\leq Y_k=X$ such that $Y_{i}/Y_{i-1}$ is cocritical for all $i=1,\dots,k$. One can verify by induction on $k$ that the vanishing of the $\tau$-local cohomology functors on all the factors of the form $Y_{i}/Y_{i-1}$ implies their vanishing on $X$. Thus we may suppose $X$ to be cocritical, in particular $E(X)\cong E(\pi)$ for some $\pi\in\Sp(\C)$. 

If $\Gdim_\tau(X)$ is not finite, then there is nothing to prove, therefore we suppose $\Gdim_\tau(X)=d<\infty$ and we proceed  by induction on $d$. If $d=-1$, then $\Gamma^n_\tau(X)=0$ for all $n>0$, by Lemma \ref{coho1}(1). Thus, $\Gamma^n_\tau(X)\neq0$ implies $n=0\leq \Gdim_\tau(X)+2=-1+2=1$. If $d>-1$, consider the following long exact sequence:
\begin{align*}
0\to \Gamma_\tau^0(X)&\to \Gamma_\tau^0(\l_\pi(X))\to \Gamma_\tau^0(\l_\pi(X)/X)\to \Gamma_\tau^1(X)\to \Gamma_\tau^1(\l_\pi(X))\to \\ 
&\to \Gamma_\tau^1(\l_\pi(X)/X)\to \dots \to \Gamma_\tau^n(X)\to \Gamma_\tau^n(\l_\pi(X))\to \Gamma_\tau^n(\l_\pi(X)/X)\to \Gamma_\tau^{n+1}(X)\to \dots
\end{align*}
Notice that $\Gamma_\tau^0(X)= \Gamma_\tau^0(\l_\pi(X))=0$, since we supposed that $d>-1$ and so, by stability, $X$ is $\tau$-torsion free. Furthermore, using (Hyp.3) and Corollary \ref{vanish_ex}, one can show that $\Gamma_\pi^n(\l_\pi(X))=0$ for all $n\in\N$. Using again that $X$ (and so $E(\pi)$) is $\tau$-torsion free, we get $\tau\preceq \pi$ and so we can apply Corollary \ref{1leq2vanish} to show that $\Gamma_\tau^n(\l_\pi(X))=0$
for all $n>0$. One obtains the following isomorphisms:
$$\Gamma_\tau^n(X)=\Gamma_\tau^{n-1}(\l_\pi(X)/X)\,,\ \ \ \forall n\geq1\,.$$
By Lemma \ref{Gdimloc<}, $\Gdim_\tau(\l_\pi(X)/X)<d$ and so we can apply our inductive hypothesis to show that $\Gamma_\tau^{n}(\l_\pi(X)/X)=0$ for all $n>\Gdim_\tau (\l_\pi(X)/X)+2$. Thus, if $\Gamma_\tau^n(X)\neq 0$ for some $n>0$, then $\Gamma_\tau^{n-1}(\l_\pi(X)/X)\neq 0$ and so $n-1\leq\Gdim_\tau (\l_\pi(X)/X)+2$, that is, $n\leq \Gdim_\tau(\l_\pi(X)/X)+3\leq \Gdim_\tau(X)+2$.
\end{proof}

\subsection{Application: exactness of products in the localization}

In the definition of Grothendieck category one assumes direct limits to be exact but no assumption is required on the exactness of products. We will see in the last part of this paper that knowing that a Grothendieck category has ``almost exact products" has very nice consequences on its derived category. In the following definition we precise the meaning of ``almost exact products":

\begin{definition}\label{quasi_exact_prod}
Let $\C$ be a Grothendieck category. For a non-negative integer $n$, $\C$ is said to satisfy the axiom (Ab.4$^*$)-$k$ if, for any set $I$ and any collection of objects $\{X_i\}_{i\in I}$,
$$\prod_{i\in I}{}^{(n)}X_i=0 \ \ \ \forall n>k\, ,$$
where $\prod_{i\in I}^{(n)}(-)$ is the $n$-th derived functor of the product $\prod:\C^{I}\to \C$. We say that $\C$ is (Ab.4$^*$) if it is (Ab.4$^*$)-$0$.
\end{definition}
In general, there is no reason for a Grothendieck category to be (Ab.4$^*$)-$k$ for any $k$. Anyway, categories of modules are (Ab.4$^*$) and it is a classical result of Gabriel and Popescu that any Grothendieck category is a quotient category of a category of modules. Thus, it seems natural to ask for sufficient conditions on a torsion theory $\tau=(\T,\F)$ in an (Ab.4$^*$) Grothendieck category $\C$ that ensure that the quotient category $\C/\T$ is (Ab.4$^*$)-$k$ for some $k\in\N$. We give two such conditions in the following lemma, a deeper criterion is given in Theorem \ref{Ab4*_quot}.

\begin{lemma}
Let $\C$ be an (Ab.4$^*$) Grothendieck category, let $\tau=(\T,\F)\in\tors(\C)$ and suppose one of the following conditions holds:
\begin{enumerate}[\rm (1)]
\item $\T$ is closed under taking products (in this case $\tau$ is said to be a TTF);
\item $\tau$ is exact.
\end{enumerate}
Then, the quotient category $\C/\T$ is still (Ab.4$^*$).
\end{lemma}
\begin{proof}
Let $I$ be a set and for each $i\in I$ consider objects $A_i,$ $B_i$ and $C_i\in\C/\T$ such that
$$0\to A_i\to B_i\to C_i\to 0 \ \ \ \text{is exact in }\ \C/\T\, .$$
Hence, for all $i\in I$, one obtains exact sequences  $0\to \S_\tau(A_i)\to \S_\tau(B_i)\to \S_\tau(C_i)\to T_i\to 0$ in $\C$,
where $T_i\in \T$. Using the (Ab.4$^*$) property in $\C$ we get an exact sequence
$$0\to \prod_{i\in I}\S_\tau(A_i)\to \prod_{i\in I}\S_\tau(B_i)\to \prod_{i\in I}\S_\tau(C_i)\to \prod_{i\in I}T_i\to 0\,.$$
If $\tau$ is a TTF, then $ \prod_{i\in I}T_i$ is $\tau$-torsion and so we can apply $\Q_{\tau}$ to the above exact sequence obtaining the following short exact sequence
$$0\to \prod_{i\in I}A_i\to \prod_{i\in I}B_i\to \prod_{i\in I}C_i\to \Q_{\tau}(\prod_{i\in I}T_i)=0\, ,$$
by Corollary \ref{prod} and the exactness of $\Q_{\tau}$. On the other hand, if $\S_\tau$ is exact, we get $T_i=0$ for all $i\in I$ above and so again one can easily conclude.
\end{proof}

%
%
%

\begin{theorem}\label{Ab4*_quot}
Let $\C$ be an (Ab.4$^*$) Grothendieck category which satisfies hypotheses (Hyp.1), (Hyp.2) and (Hyp.3), and let $\tau\in\tors (\C)$. If $\Gdim(\C/\T)=k<\infty$, then $\C/\T$ is (Ab.4$^*$)-$k+1$.
\end{theorem}
\begin{proof}
Let $\{X_i\}_{i\in I}$ be a family of objects in $\C/\T$. For all $i\in I$, choose an injective resolution $0\to \S_\tau (X_i)\to E_i^{\bullet}$ of $\S_\tau(X_i)$ in $\C$. Since $\Q_{\tau}$ is exact and sends injective objects to injective objects, the complex $\Q_{\tau}(E_i^\bullet)$ provides an injective resolution for $X_i$.  Thus, for all $n>k+1$,
\begin{align*}\prod_{i\in I}{}^{(n)}X_i&=H^n\left(\prod_{i\in I}\Q_\tau\left(E_{i}^{\bullet}\right)\right)=H^n\left(\Q_\tau\left(\prod_{i\in I}\l_\tau\left(E_{i}^{\bullet}\right)\right)\right)&\text{by Corollary \ref{prod}}\\
&=\Q_\tau\left(H^n\left(\prod_{i\in I}\l_\tau\left(E_{i}^{\bullet}\right)\right)\right)&\text{exact funct. commute with cohom.}\\
&=\Q_\tau\left(\prod_{i\in I}H^n\left(\l_\tau\left(E_{i}^{\bullet}\right)\right)\right)=\Q_\tau\left(\prod_{i\in I}{\rm R}\l^n_\tau\left(E_{i}^{\bullet}\right)\right)&\text{exact funct. commute with cohom.}\\
&=\Q_\tau\left(\prod_{i\in I}\Gamma^{n+1}_\tau\left(E_{i}^{\bullet}\right)\right)=0&\text{Lemma \ref{coho1} and Theorem \ref{vanish_loc_co_hom}.}
\end{align*}
\end{proof}

\begin{corollary}
Let $\C$ be an (Ab.4$^*$) Grothendieck category which satisfies hypotheses (Hyp.1), (Hyp.2) and (Hyp.3). If $\Gdim(\C)=k<\infty$, then $\C/\T$ is (Ab.4$^*$)-$k+1$ for all $\tau=(\T,\F)\in\tors(\C)$.
\end{corollary}

\section{Injective classes}\label{Sec_inj_class}
Let $\C$ be a Grothendieck category and let $\I$ be a class of objects of $\C$. Slightly generalizing the setting of \cite{CPS}, we say that a morphism $\phi:X\to Y$ in $\C$ is an {\em $\I$-monomorphism} if $$\hom_\C(\phi,K):\hom_\C(Y,K)\to \hom_\C(X,K)$$
is an epimorphism of Abelian groups for every $K\in\I$. We say that $\C$ has {\em enough $\I$-injectives} if every object $X$ admits an $\I$-monomorphism $X\to K$ for some $K\in \I$.

\begin{definition}\label{def.inj.class}{\rm \cite{CPS}}
A subclass $\mathcal I$ of a Grothendieck category $\C$ is an {\em injective class} (of $\C$) if it is closed under products and direct summands, and $\C$ has enough $\mathcal I$-injectives.
\end{definition}

The following lemma shows that a choice of the class of $\I$-monomorphisms is equivalent to the choice of an injective class.

\begin{lemma}
Let $\I$ be an injective class of a Grothendieck category $\C$. An object $X\in\C$ belongs to $\I$ if and only if $\hom_\C(-,X)$ sends $\I$-monomorphisms to epimorphisms of Abelian groups. 
\end{lemma}
\begin{proof}
Suppose that $\hom_\C(-,X)$ sends $\I$-monomorphisms to surjective morphisms. By definition of injective class, there is an $\I$-monomorphism $\varphi:X\to E$ for some $E\in \I$. Hence, $\hom_\C(\varphi,X):\hom_\C(E,X)\to \hom_\C(X,X)$ is surjective and so $X$ is a direct summand of $E$. Thus, $X$ belongs to $\I$. The converse is trivial. 
\end{proof}

\begin{definition}
Let $\C$ be a Grothendieck category and let $\I$ be an injective class. $\I$ is an {\em injective class of injectives} provided any object in $\I$ is an injective object. We denote by $\Inj(\C)$ the poset of all the injective classes of injectives in $\C$, where, given $\I$ and $\I'\in\Inj(\C)$
$$\I\preceq \I'\ \text{ if and only if } \ \I\subseteq \I'\,.$$ 
\end{definition}
\subsection{Examples}

Before proceeding further we give some examples of injective classes (not necessarily of injectives).
\begin{example}
Let $\C$ be a Grothendieck category. Then
\begin{enumerate}[\rm (1)]
\item $\I=0$ is an injective class. In this case every morphism is an $\I$-monomorphism;
\item $\I=\C$ is an injective class. In this case, a morphism is an $\I$-monomorphism if and only if it has a left inverse (it is a splitting monomorphism);
\item the class $\I$ of all injective objects is an injective class. With this choice, $\I$-monomorphisms are precisely the usual monomorphisms.
\end{enumerate}
\end{example}

Part (3) of the above example can be generalized as follows:

\begin{lemma}
Let $\C$ be a Grothendieck category and let $\I$ be an injective class in $\C$. Every $\I$-monomorphism is in particular a monomorphism if and only if $\I$ contains the class of injectives.
\end{lemma}

The proof of the above lemma is left to the reader.

\begin{example}
Let $\C$ be a Grothendieck category and let $\tau=(\T,\F)\in\tors(\C)$,  then $\F$ is an injective class. In fact, $\F$ is  closed under taking products and direct summands. Furthermore, given an object $X$ and $F\in\F$, apply the functor $\hom_\C(-,F)$ to the exact sequence $0\to \t_\tau(X)\to X\to X/\t_\tau(X)\to 0$, to obtain the following exact sequence of Abelian groups:
$$\hom_\C(X/\t_\tau(X),F)\to\hom_\C(X,F)\to \hom_\C(\t_\tau(X),F)\, .$$
Clearly $\hom_\C(\t_\tau(X),F)=0$ and so the canonical projection $X\to X/\t_\tau(X)$ is an $\F$-monomorphism of $X$ into an element of $\F$.
\end{example}

\begin{example}\label{ntf}
Let $R$ be a ring and let $I\leq R$ be a two-sided ideal. We claim that the class 
$$\I=\{M\in \lmod R:IM=0\}$$  
is an injective class in $\lmod R$. In fact, $\I$ is closed under products and direct summands. Furthermore, given a left $R$-module $M$, we can always consider the canonical projection $p:M\to M/IM$, where $M/IM\in\I$. We have to show that $p$ is an $\I$-monomorphism. Let $N\in \I$ and let $\phi:M\to N$ be a morphism. Given $x\in IM$, there exists $y\in M$ and $i\in I$ such that $iy=x$ and so, $\phi(x)=\phi(iy)=i\phi(y)=0$ as $IN=0$. Thus, $\phi$ factors through $p$ as desired.
%
\end{example}

We remark that an injective class does not need to satisfy any reasonable closure property but closure under products and direct summands (that are assumed in the definition). In fact, an example of an injective class that is closed nor under submodules or infinite direct sums  is given by the class of all injective modules (at least in the non-Noetherian case). An  injective class that is not closed under extensions is described in  Example \ref{ntf}.
 For further examples we refer to \cite{CPS}.

\subsection{Injective classes vs torsion theories}

\begin{definition}
Let $\C$ be a Grothendieck category and let $\tau=(\T,\F)\in\tors(\C)$. We define the following subclass of $\C$:
$$\I_\tau=\{\text{injective objects in $\F$}\}\,.$$
\end{definition}

It is possible to give quite an explicit characterization of those morphisms that are $\I_\tau$-monomorphism:
\begin{lemma}\label{i-mono}
Let $\C$ be a Grothendieck category, let $\tau=(\T,\F)\in\tors(\C)$ and let $\phi:X\to Y$ be a morphism in $\C$. The following are equivalent:
\begin{enumerate}[\rm (1)]
\item $\phi$ is an $\I_\tau$-monomorphism;
\item $\hom_\C(\phi, E)$ is an epimorphism for any injective object $E$ which cogenerates $\tau$;
\item $\ker(\phi)$ is $\tau$-torsion;
\item $\Q_\tau(\phi)$ is a monomorphism.
\end{enumerate}
\end{lemma}
\begin{proof}
The implication (1)$\Rightarrow$(2) is trivial since an injective object which cogenerates $\tau$ necessarily belongs to $\I_\tau$. Let us prove the implication (2)$\Rightarrow$(3). Choose an injective cogenerator $E$ for $\tau$ and apply the functor $\hom_\C(-,E)$ to the exact sequence $0\to \ker(\phi)\to X\to Y$ obtaining the following exact sequence of Abelian groups
$$\hom_\C(Y,E)\to \hom_\C(X,E)\to \hom_\C(\ker(\phi), E)\to0\, .$$
If $\hom_\C(\phi,E)$ is an epimorphism then $\hom_\C(\ker(\phi),E)=0$ that is, $\ker(\phi)$ is $\tau$-torsion. The equivalence (3)$\Leftrightarrow$(4) follows by the exactness of $\Q_\tau$ and the fact that $\ker(\Q_\tau)=\T$. It remains only to prove that (3)$\Rightarrow$(1). Indeed, given $K\in\I_\tau$ we can obtain as before an exact sequence $\hom_\C(Y,K)\to \hom_\C(X,K)\to \hom_\C(\ker(\phi), K)$. Since $K$ is $\tau$-torsion free and $\ker(\phi)$ is $\tau$-torsion, we have that $\hom_\C(\ker(\phi), K)=0$, as desired.
\end{proof}

\begin{lemma}
Let $\C$ be a Grothendieck category and let $\tau=(\T,\F)\in\tors(\C)$. Then, $\I_\tau\in \Inj(\C)$.
\end{lemma}
\begin{proof}
$\I_\tau$ is the intersection of $\F$ with the class $\E_\C$ of all the injective objects in $\C$. The closure properties of $\F$ and $\E_\C$ imply that $\I_\tau$ is closed under products and direct summands. It remains to show that $\C$ has enough $\I_\tau$-injectives. Indeed, let $X$ be an object of $\C$
and let $\phi:X\to E(X/\tor_\tau(X))$ be the composition of the canonical morphisms $X\to X/\t_\tau(X)$ and $X/\t_\tau(X)\to E(X/\t_\tau(X))$. Clearly $E(X/\tor_\tau(X))\in \I_\tau$ and, furthermore, the kernel of $\phi$ is precisely $\t_\tau(X)$. By Lemma \ref{i-mono}, $\phi$ is an $\I_\tau$-monomorphism.
\end{proof}

Let $\C$ be a Grothendieck category. By the above lemma, we can associate an injective class of injectives to any given torsion theory. Let now $\I\in\Inj(\C)$ and define 
$$\F_\I=\{\text{sub-objects of the elements of $\I$}\}\,.$$

\begin{lemma}\label{cl-hyp-inj}
Let $\C$ be a Grothendieck category and let $\I\in\Inj(\C)$. Then, $\F_\I$ is a torsion free class.
\end{lemma}
\begin{proof}
Closure under taking sub-objects, products and injective envelopes easily follow by construction and the closure hypotheses on $\I$. It remains to prove that $\F_\I$ is closed under taking extensions. Let $X$ be an object in $\C$ and let $Y\leq X$ be a sub-object such that both $Y$ and $X/Y\in\F_\I$. By construction, there exist $I_1$ and $I_2\in \I$ such that $Y\leq I_1$ and $X/Y\leq I_2$. Let $\phi_1:X\to I_1$ be a morphism extending the canonical inclusion $Y\to I_1$ (whose existence is ensured by the injectivity of $I_1$) and let $\phi_2:X\to I_2$ be the composition of the canonical projection $X\to X/Y$ with the inclusion $X/Y\to I_2$. Define $\phi=\phi_1\oplus \phi_2:X\to I_1\oplus I_2$. Then, $\ker(\phi)=\ker(\phi_1)\cap \ker(\phi_2)=0$ and so $X$ is a sub-object of $I_1\oplus I_2\in \I$.
\end{proof}

\begin{definition}
Let $\C$ be a Grothendieck category and let $\I\in\Inj(\C)$. We define $\tau_\I$ to be the unique torsion theory on $\C$ whose torsion free class is $\F_\I$.
\end{definition}

We are now ready to prove the main result of this subsection:

\begin{theorem}\label{corresp_inj_tt}
Let $\C$ be a Grothendieck category. Then the map $\tau\mapsto\I_\tau$ is an order-reversing bijection between $\tors(\C)$ and $\Inj(\C)$. The inverse bijection is given by the correspondence $\I\mapsto\tau_\I$.
\end{theorem}
\begin{proof}
Let $\tau=(\T,\F)\in\tors(\C)$ we want to prove that $\F_{\I_{\tau}}=\F$. The inclusion $\F_{\I_\tau}\subseteq \F$ is trivial, while, given $F\in \F$ and an injective cogenerator $E$ for $\tau$, there exists a set $S$ such that $F\leq E^{S}\in \I_\tau$ and so $F\in \F_{\I_\tau}$. 

\smallskip
On the other hand, let $\I\in\Inj(\C)$ and $\tau_\I=(\F,\T)$. We want to prove that $\I_{\tau_\I}=\I$. The inclusion $\I\subseteq \I_{\tau_\I}$ is trivial, while, given $I\in\I_{\tau_\I}$, by definition $I\in \F=\F_{\I}$ and so $I$ is an injective sub-object, and so a summand, of an element of $\I$, thus $I\in \I$.  
\end{proof}

The above theorem together with Theorem \ref{class_gabriel} gives the following

\begin{corollary}\label{coro_inj}
Let $\C$ be a Grothendieck category satisfying (Hyp.1). There are bijections
$$\xymatrix{
\{\text{Gen. closed subsets of }\Sp(\C)\} \ar@<2,5pt>[r] & \Inj(\C)\ar@<2,5pt>[l]\ar@<2,5pt>[r] &\{\text{Spec. closed subsets of }\Sp(\C)\}\ar@<2,5pt>[l]
}$$
$$\xymatrix{
G(\tau_\I)&&\I\ar@{|->}[ll]\ar@{|->}[rr]&&S(\tau_\I)
}$$
\end{corollary}


\subsection{Module categories}

In this subsection we specialize our results about injective classes to categories of modules, re-obtaining as corollaries the main results of \cite{CPS}.

\begin{definition}{\rm \cite{Dauns, CPS}}
Let $R$ be a ring. A non-empty set $\A$ of left ideal of $R$ is said to be a {\em torsion free set} (or, {\em saturated set}) if the following conditions hold:
\begin{enumerate}[\rm ({NS}.1)]
\item $\A$ is closed under arbitrary intersections;
\item for all $x\in R$ and $I\in \A$, $(I:x)=\{r\in R:rx\in I\}\in \A$;
\item if a proper left ideal $J< R$ has the property that, for all $x\in R\setminus J$, there is $I\in \A$, such that $(J:x)\subseteq I$, then $J\in \A$.
\end{enumerate}
\end{definition}

Let us recall the following fact from \cite{Dauns}.

\begin{lemma}\label{Dauns}{\rm\cite[Corollary 2.3.14]{Dauns}}
Let $R$ be a ring. There is a bijective correspondence between $\tors(\lmod R)$ and the family of torsion free sets of ideals of $R$.
\end{lemma}

The following corollary, which is an immediate consequence of Lemma \ref{Dauns} and Theorem \ref{corresp_inj_tt}, is a generalization of \cite[Theorem 3.7]{CPS}.

\begin{corollary}\label{sat_sat}
Let $R$ be a ring. There is a bijective correspondence between $\Inj(\lmod R)$ and the family of torsion free sets of ideals of $R$.
\end{corollary}

Using the classical theory of injective modules over commutative Noetherian rings we can also derive from Corollary \ref{coro_inj} the characterization of injective classes of injectives over such rings given in \cite[Corollary 3.9]{CPS}. 

\begin{definition}
Let $R$ be a commutative ring. The {\em spectrum} $\Spec(R)$ of $R$ is the poset of all the prime ideals in $R$ (ordered by inclusion). Given $\frak p,\frak q\in \Spec(R)$, if $\mathfrak p\subseteq \mathfrak q$ we say that $\mathfrak p$ is a {\em generalization} of $\mathfrak q$ and that $\mathfrak q$ is a {\em specialization} of $\mathfrak p$.
\end{definition}

\begin{lemma}
Let $R$ be a commutative Noetherian ring, then there is a bijection
$$\Sp(\lmod R)\to \Spec(R)\,, \ \ \mathfrak p\mapsto E(R/\mathfrak p)\,.$$
Furthermore, given $\mathfrak p$, $\mathfrak q\in\Spec(R)$ and denoting by $\pi(\mathfrak p)$ and $\pi(\mathfrak q)$ the prime torsion theories cogenerated by $E(R/\mathfrak p)$ and $E(R/\mathfrak q)$ respectively, 
$$(\mathfrak p\subseteq \mathfrak q)\ \ \Longleftrightarrow \ \ (\pi(\mathfrak q)\preceq\pi(\mathfrak p))\,.$$
\end{lemma}
\begin{proof}
The fact that this map is well-defined and bijective is Proposition 3.1 in \cite{matlis}. Furthermore, if $\mathfrak p\subseteq \mathfrak q$, then there is a non-trivial map $R/\mathfrak p\to R/\mathfrak q$ and this can be used to show that $\hom_R(E(R/\mathfrak p),E(R/\mathfrak q))\neq 0$. By Lemma \ref{stab_ort}, this shows that $\pi(\mathfrak q)\preceq\pi(\mathfrak p)$. On the other hand, if $\pi(\mathfrak q)\preceq\pi(\mathfrak p)$, then $E(R/\mathfrak p)$ is not $\pi(\mathfrak q)$-torsion and so, by stability, it is $\pi(\mathfrak q)$-torsion free. In particular, $\hom_R(R/\mathfrak p, E(R/\mathfrak q))\neq 0$. Consider a non-trivial morphism $\phi: R/\mathfrak p\to E(R/\mathfrak q)$ and let $0\neq x\in \Im(\phi)$. Then, $\mathfrak p\subseteq \Ann_R(x)\subseteq \mathfrak q$, since $\Ann_R(x)$ is a $\mathfrak q$-primary ideal by Lemma 3.2 in \cite{matlis}.
\end{proof}

The following corollary is a consequence of the above lemma and Corollary \ref{coro_inj}.

\begin{corollary}\label{sat_noeth}{\rm\cite[Corollary 3.9]{CPS}}
Let $R$ be a commutative Noetherian ring. There is a bijective correspondence between $\Inj(\lmod R)$ and the family generalization closed sets of prime ideals of $R$.
\end{corollary}

\section{Model approximation}\label{Sec_mod_approx}
In the first part of this section we recall some definitions and terminology about model categories and model approximations. 

\begin{definition}
Let $\C$ be a category and let $\W$ be a collection of morphisms in $\C$. The pair $(\C,\W)$ is a {\em category with weak equivalences} if, given two composable morphisms $\f$ and $\psi$, whenever two elements of $\{\f,\psi,\psi\f\}$ belong to $\W$ so does the third. The elements of $\W$ are called {\em weak equivalences} of $(\C,\W)$.
\end{definition}

We now recall the definition of a model category. We will give just one concrete example of model category (see Example \ref{inj_mod_struc}), we refer to \cite{Dwyer} and \cite{hovey} for further examples and properties of model categories. 

\begin{definition}\label{model_cat}
Let $\M$ be a bicomplete category and let $\W$, $\cal B$ and $\cal C$ be three classes of morphisms; $(\M,\W,\cal B,\cal C)$ is a {\em model category} provided the following conditions hold:
\begin{enumerate}[\rm({MC}.1)]
\item $(\C,\W)$ is a category with weak equivalences;
\item $\W$, $\cal B$ and $\cal C$ are closed under retracts (in the category of morphisms). That is, given a commutative diagram as follows:
$$\xymatrix{
X\ar@/_-10pt/@{.>}[rr]|{\id}\ar[r]\ar[d]_(.43)\phi&X'\ar[r]\ar[d]|(.43){\phi'}&X\ar[d]^(.43)\phi\\
Y\ar[r]\ar@/_10pt/@{.>}[rr]|\id&Y'\ar[r]&Y
}$$
if $\phi'$  belongs to $\W$ (resp., $\cal B$ or $\cal C$), so does $\phi$;
\item consider the following diagram,
$$\xymatrix{
C\ar[r]\ar[d]_c&B\ar[d]^b\\
C'\ar@{.>}[ru]|{\psi}\ar[r]&B'
}$$
where $b\in \cal B$ and $c\in\cal C$. If the external square commutes and  either $b\in \W$ or $c\in \W$, then there exists $\psi$ as above making the entire diagram commutative;
\item given a morphism $\phi$, there exist $b\in \cal B\cap \W$, $c\in \cal C$, $b'\in \cal B$ and $c'\in \cal C\cap \W$, such that
$$\phi=bc\ \text{ and }\ \phi=b'c'\,.$$
\end{enumerate}
The elements of $\W$, $\cal B$, $\cal B\cap \W$, $\cal C$ and $\cal C\cap \W$ are called respectively {\em weak equivalences, fibrations, acyclic fibrations, cofibrations and acyclic cofibrations}.\\ Given an object $X\in\M$, if the unique map from the initial object to $X$ is a cofibration, then $X$ is said to be {\em cofibrant}. If the unique map from $X$ to the terminal object is a fibration then $X$ is said to be {\em fibrant}.
\end{definition}

Recall that a morphism of cochain complexes is a quasi-isomorphism if it induces isomorphism in cohomology. Equivalently, it's mapping cone is an exact complex.

\begin{example}\label{inj_mod_struc}
Let $\C$ be a Grothendieck category and recall that the category $\Ch(\C)$ of (unbounded) cochain complexes on $\C$ is a bicomplete category. Let $\W$ be the class of quasi-isomorphisms in $\Ch(\C)$, then $(\Ch(\C),\W)$ is a category with weak equivalences.\\
Recall that a complex $E^\bullet\in\Ch(\C)$ is {\em dg-injective} if $E^n$ is injective for all $n\in\Z$ and if the complex $\mathcal{H}{om}(X^\bullet,E^\bullet)$ is exact, for any exact complex $X^\bullet\in\Ch(\C)$. Let $\mathcal B$ be the class of all the epimorphisms with dg-injective kernels and let $\mathcal C$ be the class of monomorphisms, then $(\Ch(\C),\W,\cal B,\cal C)$ is a model category (see for example \cite{hovey} or \cite{Gil07} for a proof).
\end{example}

We now introduce model approximations. Here we just give the definition and few comments, for a systematic study of the properties and applications of model approximations we refer to \cite{PS}.

\begin{definition}\label{mod_appr} Let $(\C,\W_\C)$ be a category with weak equivalences. A {\em right model approximation} for $(\C,\W_\C)$ is a model category $(\M,\W,\cal B,\cal C)$ and a pair of functors 
$$l :  \xymatrix{\C\ar@<2pt>[r]\ar@<-2pt>@{<-}[r]&\M}  : r$$ 
satisfying the following conditions:
\begin{enumerate}[\rm ({MA}.1)]
\item $l$ is left adjoint to $r$;
\item if $\f\in\W_\C$, then $l(\f)\in\W$; 
\item if $\psi$ is a weak equivalence between fibrant objects in $\M$, then $r(\psi)\in\W_\C$;
\item if $l(X)\to Y$ is a weak equivalence in $\M$ with $X$ fibrant, the adjoint morphism $X \to r(Y)$ is in $\W_\C$.
\end{enumerate}
\end{definition}

Given two categories with weak equivalences $(\C_1,\W_1)$ and $(\C_2,\W_2)$, one can find conditions under which a model approximation of $(\C_2,\W_2)$ gives automatically a model approximation of $(\C_1,\W_1)$:

\begin{lemma}\label{compatible_lemma}
Let $(\C_1,\W_1)$ and $(\C_2,\W_2)$ be categories with weak equivalences and let 
$$l :  \xymatrix{\C_1\ar@<2pt>[r]\ar@<-2pt>@{<-}[r]&\C_2}  : r$$
be an adjunction such that:
\begin{enumerate}[\rm (1)]
\item $\phi\in\W_1$ implies $l\phi\in \W_2$;
\item $\psi\in\W_2$ implies $r\psi\in \W_1$;
\item if a morphism $l(C)\to D$ is in $\W_2$, then the adjoint morphism $C\to r(D)$ is in $\W_1$.
\end{enumerate}
If $(\M,\W,\cal B,\cal C)$ is a model category and $l':\C_2\leftrightarrows \M:r'$ is a model approximation, then 
$$L: \xymatrix{\C_1\ar@<2pt>[r]\ar@<-2pt>@{<-}[r]&\M}  :R$$
 is a model approximation, where $L=l'\circ l$ and $R=r\circ r'$.
\end{lemma}
\begin{proof}
We have to verify conditions (MA.1)-(MA.4):

\noindent(MA.1)  is Theorem 1 in Section 8 of Chapter 1 of \cite{McL}.

\noindent(MA.2) If $\phi\in\W_1$, then $l\phi\in \W_2$ by our hypothesis (1). Apply the definition of model approximation to obtain that $L(\phi)=l'(l(\phi))$ is a weak equivalence in $\M$.

\noindent(MA.3) If $\psi$ is a weak equivalence between fibrant objects in $\M$, then $r'(\psi)\in\W_2$, by the definition of model approximation. Apply  (2) to obtain that $R\psi=r(r'(\psi))\in\W_1$. 

\noindent(MA.4) Let $X\in\M$ be fibrant, let $Y\in\C_1$ and let $X\to L(Y)$ be a weak equivalence in $\M$. The adjoint (under the adjunction $(l',r')$) morphism $r'(X)\to l(Y)$ belongs to $\W_2$, by the definition of model approximation. Using (3), the adjoint  (under the adjunction $(l,r)$) morphism $r(r'(X))\to Y$ belongs to $\W_1$.   
\end{proof}

The above lemma provides a motivation for the following

\begin{definition}
Let $(\C_1,\W_1)$ and $(\C_2,\W_2)$ be two categories with weak equivalences. An adjunction $l:\C_1\leftrightarrows \C_2:r$ is said to be {\em compatible with weak equivalences} if conditions (1), (2) and (3) in Lemma \ref{compatible_lemma} are satisfied.  
\end{definition}

\subsection{A relative injective model approximation}

Let $\C$ be a Grothendieck category and let $\tau=(\T,\F)\in\tors(\C)$. We extend the $\tau$-quotient and the $\tau$-section functors to categories of complexes applying them compont-wise:
\begin{equation}\label{adJ} \Q_{\tau}:\xymatrix{\Ch(\C)\ar@<2pt>[r]\ar@<-2pt>@{<-}[r]&\Ch(\C/\T)}:\S_\tau\, .\end{equation}
We use the same symbols for these new functors, it is an exercise to show that they are adjoint. 

\begin{definition}
Let $\C$ be a Grothendieck category and let $\I\in\Inj(\C)$. A morphism $\phi^\bullet$ of cochain complexes is an {\em $\I$-quasi-isomorphism} provided $\hom_{\C}(\phi^\bullet,K)$ is a quasi-isomorphism of complexes of Abelian groups for all $K\in \I$.
\\
Given $\tau=(\T,\F)\in \tors(\C)$, we define the following class of morphisms in $\Ch(\C)$
$$\W_\tau=\{\phi^\bullet : H^n(\mathrm{cone}(\phi^\bullet))\in \T\,,\ \forall n\in\mathbb Z\}\,.$$ 
\end{definition}

Recall that the mapping cone construction commutes with any additive functor (this can be easily verified by hand). This is used repeatedly in the following lemma.

\begin{lemma}\label{coni}
Let $\C$ be a Grothendieck category, let $\tau=(\T,\F)\in\tors(\C)$ and denote by $\overline\W$ be the class of quasi-isomorphisms in $\Ch(\C/\T)$. The following are equivalent for a morphism $\phi^\bullet$ in $\Ch(\C)$:
\begin{enumerate}[\rm (1)]
\item $\phi^\bullet\in \W_\tau$;
\item $\Q_{\tau}(\phi^\bullet)\in \overline\W$;
\item $\phi^\bullet$ is an $\I_\tau$-quasi-isomorphism.
\end{enumerate}
Furthermore, $(\Ch(\C),\W_\tau)$ is a category with weak equivalences.
\end{lemma}
\begin{proof}
Since $\Q_\tau$ is exact, $H^n$ and $\Q_\tau$ commute. Thus, for all $n\in\N$:
$$\Q_{\tau}(H^n(\mathrm{cone}(\phi^\bullet)))\cong H^n(\mathrm{cone}(\Q_{\tau}(\phi^\bullet)))\,.$$
This proves the equivalence between (1) and (2).

\smallskip\noindent For the equivalence between (1) and (3), notice that $\hom_\C(\phi^\bullet,K)$ is a quasi-isomorphism for all $K\in\I_\tau$ if and only if, for all $n\in\Z$ and $K\in\I_\tau$,
$$0=H^n(\cone(\hom_\C(\phi^\bullet,K)))\cong H^n(\hom_\C(\cone(\phi^\bullet),K))\cong\hom_\C(H^n(\cone(\phi^\bullet)),K)\,.$$
Thus, $\hom_\C(\phi^\bullet,K)$ is a quasi-isomorphism for all $K\in\I_\tau$ if and only if  $H^n(\cone(\phi^\bullet))\in {}^{\perp}(\I_\tau)=\T$, for all $n\in\Z$.
\end{proof}

%
The following theorem answers part (1) of Question \ref{quest_iniziale} in full generality:

\begin{theorem}\label{main_general}
Let $\C$ be a Grothendieck category and let $\tau=(\T,\F)\in\tors(\C)$. Consider the category with weak equivalences $(\Ch(\C),\W_\tau)$ and the injective model category $(\Ch(\C/\T), \overline\W,\mathcal B,\mathcal C)$ defined as in Example \ref{inj_mod_struc}. Then,  the adjunction 
$$\Q_\tau :  \xymatrix{(\Ch(\C),\W_\tau)\ar@<2pt>[r]\ar@<-2pt>@{<-}[r]&(\Ch(\C/\T), \overline\W,\mathcal B,\mathcal C)}  : \S_\tau$$ 
is a model approximation. Furthermore, the homotopy category relative to this model approximation is naturally isomorphic to the unbounded derived category ${\bf D}(\C/\T)$.
\end{theorem}
\begin{proof}
We have to verify conditions (MA.1)--(MA.4). Condition (MA.1) just states that $(\Q_\tau,\S_\tau)$ is an adjunction, while (MA.2), (MA.3) and (MA.4) are consequences of Lemma \ref{coni}. The statement about the homotopy category follows by the explicit construction given in Proposition 5.5 of \cite{PS} and the fact that $\Q_\tau$ is essentially surjective.
\end{proof}

\subsection{Towers of models}\label{towers}
In this subsection we recall the construction of the category of towers introduced in \cite{CPS1}.

\begin{definition}
Let $\M_\bullet=\{\M_n:n\in\N\}$ be a sequence of categories connected with adjunctions 
$$l_{n+1}:  \xymatrix{\M_{n+1}\ar@<2pt>[r]\ar@<-2pt>@{<-}[r]&\M_{n}} :r_n\,.$$ 
The {\em category of towers on $\M_\bullet$}, $\Tow(\M_\bullet)$ is defined as follows:
\begin{enumerate}[\rm --]
\item  an {\em object} is a pair $(a_\bullet, \alpha_\bullet)$, where $a_\bullet=\{a_n\in \M_n:n\in\N\}$ is a sequence of objects one for each $\M_n$, and $\alpha_\bullet=\{\alpha_{n+1}\colon a_{n+1}\to r_n(a_n):n\in\N\}$ is a sequence of morphisms;
\item a {\em morphism} $f_\bullet:(a_\bullet,\alpha_\bullet)\to (b_\bullet,\beta_\bullet)$ is a sequence of morphisms $f_\bullet=\{f_n\colon a_n\to b_n:n\in\N\}$ such that $r_n(f_n)\circ\alpha_{n+1}=\beta_{n+1}\circ f_{n+1}$, for all $n\in\N$.
\end{enumerate}
\end{definition}

If each $\M_n$ in the above definition is a bicomplete category, then one can construct limits and colimits component-wise in $\Tow(\M_\bullet)$, so, under these hypotheses, the category of towers is bicomplete. 

\begin{proposition}\label{model_torre}{\rm \cite[Proposition 2.3]{CPS1}}
Let $\M_\bullet=\{(\M_n,\W_n,\mathcal B_n,\mathcal C_n):n\in\N\}$ be a sequence of model categories connected with adjunctions 
$$l_{n+1}:  \xymatrix{\M_{n+1}\ar@<2pt>[r]\ar@<-2pt>@{<-}[r]&\M_{n}} :r_n$$ 
and suppose that each $r_n$ preserves fibrations and acyclic fibrations. Define the following classes of morphisms in $\Tow(\M_\bullet)$:
\begin{enumerate}[\rm --]
\item $\W_\Tow=\{f_\bullet:f_n\in\W_n \,,\ \forall n\in\N\}$;
\item $\mathcal B_\Tow=\{f_\bullet:f^*_n\in\mathcal B_n \,,\ \forall n\in\N\}$, where $f^*_n$ is constructed as follows. First we define an object $(p_\bullet, \pi_\bullet)$ in $\Tow(\M_\bullet)$ where each $p_n$ comes from a pull-back diagram
$$\xymatrix{
p_{n}\ar@{}[rrd]|{\rm P.B.}\ar[d]_{\bar\beta_n}\ar[rr]^{\bar f_{n-1}}& &    b_n\ar[d]^{\beta_n}\\
r_{n-1}(a_{n-1})\ar[rr]_{r_{n-1}(f_{n-1})}& &         r_{n-1}(b_{n-1})
}$$
and $\pi_n=r_{n-1}(f_{n-1})\circ\bar\beta_n$. Then,  $f^*_n:a_n\to p_n$ is defined, using the universal property of the pull-back, as the unique morphism such that $\bar\beta_nf^*_n=\alpha_n$ and $\bar f_{n-1}f^*_n=f_n$. 
\item $\mathcal C_\Tow=\{f_\bullet:f_n\in\mathcal C_n \,,\ \forall n\in\N\}$
\end{enumerate}
Then, $(\Tow(\M_\bullet),\W_\Tow,\mathcal B_\Tow,\mathcal C_\Tow)$ is a model category.
\end{proposition}

Let us introduce the specific sequence of model categories we are interested in:

\begin{lemma}\label{model_troncati}
Let $\C$ be a Grothendieck category, let $\tau=(\T,\F)\in\tors(\C)$ and let $n\in\N_+$. There is a model category $(\Ch^{\geq -n}(\C),\W^{\geq -n}_\tau,\mathcal B^{\geq -n}_\tau,\mathcal C^{\geq -n}_\tau)$, where:
\begin{enumerate}[\rm --]
\item $\W^{\geq -n}_\tau=\{\phi^\bullet:\phi^\bullet \text{ is a $\tau$-quasi-isomorphism}\}$;
\item $\mathcal B^{\geq -n}_\tau=\{\phi^\bullet:\phi^i \text{ is an epimorphism and $\ker(\phi^i)\in\I_\tau$ for all $i\geq-n$}\}$;
\item $\mathcal C^{\geq -n}_\tau=\{\phi^\bullet:\phi^i\text{ is a $\tau$-monomorphism for all $i\geq -n$}\}$.
\end{enumerate}
For all $n\in\Z$, the above choice of weak equivalences, fibrations and cofibrations makes\linebreak $(\Ch^{\geq-n}(\C),\W^{\geq -n}_\tau,\mathcal B^{\geq -n}_\tau,\mathcal C^{\geq -n}_\tau)$ into a model category.\\
Furthermore, there is an adjunction
\begin{equation}\label{adj_n}
l_{n+1}:  \xymatrix{\Ch^{\geq -n-1}(\C)\ar@<2pt>[r]\ar@<-2pt>@{<-}[r]&\Ch^{\geq -n}(\C)} :r_n\,,
\end{equation}
where $l_{n+1}$ is the obvious inclusion while $r_n$ is the truncation functor. In this situation, $r_n$ preserves fibrations and acyclic fibrations.
\end{lemma}
\begin{proof}
The proof can be obtained exactly as in the case when $\C$ is a category of modules, see \cite[Theorem 1.9]{CPS2}.
\end{proof}

%

\begin{definition}
Let $\C$ be a Grothendieck category and let $\tau=(\T,\F)\in\tors(\C)$. Consider the sequence $\M_\bullet=\{(\Ch^{\geq -n}(\C),\W^{\geq -n}_\tau,\mathcal B^{\geq -n}_\tau,\mathcal C^{\geq -n}_\tau):n\in\N\}$ defined in Lemma \ref{model_troncati}. We denote the category $(\Tow(\M_\bullet),\W_\Tow,\mathcal B_\Tow,\mathcal C_\Tow)$ constructed as in Proposition \ref{model_torre} by $(\Tow_\tau(\C),\W_\Tow,\mathcal B_\Tow,\mathcal C_\Tow)$. Furthermore, we denote by $\Tow:\Ch(\C)\to \Tow_\tau(\C)$ the so-called {\em tower functor}, which sends a complex $X^\bullet$ to the sequence of its successive truncations $\dots\to X^{\geq -n}\to X^{\geq -n+1}\to \dots\to X^{\geq 0}$ and acts on morphisms in the obvious way (see \cite{CPS1}).\\
If $\tau=(0,\C)$ is the trivial torsion theory we denote $\Tow_\tau(\C)$ by $\Tow(\C)$. 
\end{definition}
A typical object $X^{\bullet}_\bullet$ of $\Tow_\tau(\C)$ is a commutative diagram of the form
$$\xymatrix{
\vdots\ar[d] & \vdots\ar[d] & \vdots\ar[d] & \vdots\ar[d] & \vdots\ar[d] & \vdots\ar[d] \\
0\ar[r] & X_{2}^{-2}\ar[d]\ar[r]^{d^{-2}_{2}} & X_{2}^{-1}\ar[d]^{t^{-1}_{2}} \ar[r]^{d^{-1}_{2}}& X_{2}^{0}\ar[d]^{t^{0}_{2}} \ar[r]^{d^{0}_{2}}& X_{2}^{1}\ar[d]^{t^{1}_{2}}\ar[r]^{d^{1}_{2}} & X_{2}^{2} \ar[d]^{t^{2}_{2}}\ar[r]^{d^{2}_{2}}& \dots \\
   & 0 \ar[r] & X_{1}^{-1}\ar[d] \ar[r]^{d^{-1}_{1}}& X_{1}^{0}\ar[d]^{t^{0}_{1}}\ar[r]^{d^{0}_{1}} & X_{1}^{1}\ar[d]^{t^{1}_{1}} \ar[r]^{d^{1}_{1}}& X_{1}^{2}\ar[d]^{t^{2}_{1}} \ar[r]^{d^{2}_{1}}& \dots \\
   &     & 0\ar[r]  & X_{0}^{0}\ar[r]^{d^{0}_{0}} & X_{0}^{1}\ar[r]^{d^{1}_{0}} & X_{0}^{2}\ar[r]^{d^{2}_{0}} & \dots 
}$$
where $(X^\bullet_n,d_n^\bullet)$ is a cochain complex for all $n\in\N$, and $X^m_n=0$ for all $m<-n$.  

\subsubsection{Approximation of (Ab.4$^*$)-$k$ categories}

Let $\C$ be a Grothendieck category and let $\tau=(\T,\F)\in\tors(\C)$. The category $\Tow_\tau(\C)$ can be seen as a full subcategory of the category ${\rm Func}(\N,\Ch(\C))$ of functors $\N\to \Ch(\C)$ and so we can restrict the usual limit functor to obtain a functor $\lim:\Tow_\tau(\C)\to \Ch(\C)$. 

\smallskip
In \cite{CPS1} and \cite{CPS2} the authors show that when $\C$ is a category of modules over a commutative Noetherian ring of finite Krull dimension, 
\begin{equation}\label{tower_adj_to}
\Tow:  \xymatrix{(\Ch(\C),\W_\tau)\ar@<2pt>[r]\ar@<-2pt>@{<-}[r]&(\Tow_\tau(\C),\W_\Tow,\mathcal B_\Tow,\mathcal C_\Tow)} :\lim\,,\end{equation}
is a model approximation for all $\tau\in \tors (\C)$. On the other hand, if the Krull dimension is not finite, one can always find counterexamples. In what follows we try to better understand this kind of construction when $\C$ is a general Grothendieck category. First of all, notice that when we construct the homotopy category $\bf D(\C/\T)$ inverting the weak equivalences in $(\Ch(\C),\W_\tau)$ we are really doing two things at the same time:
\begin{enumerate}[\rm (1)]
\item localize complexes over $\C$ to complexes over $\C/\T$;
\item pass from a category of complexes over $\C/\T$ to its derived category.
\end{enumerate}
Our strategy is to separate the two operations in two different ``steps", where each ``step" corresponds to a pair of adjoint functors. The composition of these adjunctions is our candidate for a model approximation, as we will see in Theorem \ref{main_th}. Before that, we need to recall some facts about homotopy limits as defined in \cite{BN}. In fact, when $\C$ is an (Ab.4$^*$) Grothendieck category,  one can prove as in \cite[Application 2.4]{BN} that for each $X^\bullet\in \Ch(\C)$ there is a quasi-isomorphism
\begin{equation}\label{lim_of_trunc}X^\bullet\cong {\mathrm{holim}}X^{\geq -i}\, .\end{equation}
This formula is useful as it allows to reduce many questions to half-bounded complexes. On the other hand, for \eqref{lim_of_trunc} to hold, it is sufficient that the ambient category is (Ab.4$^*$)-$k$ for some finite $k$:
\begin{theorem}\label{theo_4*K}{\rm\cite[Theorem 1.3]{HX}}
Let $\C$ be a Grothendieck and assume that $\C$
satisfies (Ab.4$^*$)-$k$ for some positive integer $k$. Then, for every $X^\bullet \in \Ch(\C)$, we
have an isomorphism $X^\bullet \cong \mathrm{holim}X^{\geq-i}$ in $\bf D(\C)$.
\end{theorem}

We are now ready to prove our main result.

%

\begin{theorem}\label{main_th}
Let $\C$ be a Grothendieck category and let $\tau=(\T,\F)\in\tors(\C)$ be a torsion theory such that $\C/\T$ is ($Ab.4^*$)-$k$ for some positive integer $k$. Then, the composition of the following adjunctions
$$\xymatrix{(\Ch(\C), \W_\tau)\ar@<2pt>[r]^{\Q_\tau}\ar@<-2pt>@{<-}[r]_{\S_\tau}& (\Ch(\C/\T),\overline\W)\ar@<2pt>[r]^(.37){\Tow}\ar@<-2pt>@{<-}[r]_(.37){\lim} &(\Tow(\C/\T),\W_\Tow,\mathcal B_\Tow,\mathcal C_\Tow)}$$
is a model approximation.
\end{theorem}
\begin{proof}
By Lemma \ref{compatible_lemma}, it is enough to show that $(\Q_\tau,\S_\tau)$ is compatible with weak equivalences (and this follows by Lemma \ref{coni}) and that $(\Tow,\lim)$ is a model approximation. The proof that $(\Tow,\lim)$ is a model approximation is given in \cite{CPS1} for categories of modules. One can follow that proof almost without changes, using Theorem \ref{theo_4*K} (that applies here since $\C/\T$ is ($Ab.4^*$)-$k$) instead of Application 2.4 in \cite{BN}. 
\end{proof}

Combining the above theorem with Theorem \ref{Ab4*_quot} we obtain the following corollary, which gives a partial answer to part (2) of Question \ref{quest_iniziale}:

\begin{corollary}\label{main_co}
Let $\C$ be a Grothendieck category satisfying (Hyp.1), (Hyp.2) and (Hyp.3), and let $\tau\in\tors(\C)$. If $\Gdim_\tau(\C)<\infty$, then the composition 
$$\xymatrix{(\Ch(\C), \W_\tau)\ar@<2pt>[r]^{\Q_\tau}\ar@<-2pt>@{<-}[r]_{\S_\tau}& (\Ch(\C/\T),\overline\W)\ar@<2pt>[r]^(.37){\Tow}\ar@<-2pt>@{<-}[r]_(.37){\lim} &(\Tow(\C/\T),\W_\Tow,\mathcal B_\Tow,\mathcal C_\Tow)}$$
is a model approximation.
\end{corollary}

%
%
%

\bibliographystyle{amsplain}
\bibliography{/Users/simonevirili/Dropbox/refs}

\providecommand{\bysame}{\leavevmode\hbox to3em{\hrulefill}\thinspace}
\providecommand{\MR}{\relax\ifhmode\unskip\space\fi MR }
\providecommand{\MRhref}[2]{%
  \href{http://www.ams.org/mathscinet-getitem?mr=#1}{#2}
}
\providecommand{\href}[2]{#2}
\begin{thebibliography}{10}

\bibitem{Albu-Nastasescu2}
Toma Albu and Constantin N{\u{a}}st{\u{a}}sescu, \emph{{Some aspects of
  non-noetherian local cohomology.}}, Commun. Algebra \textbf{8} (1980),
  1539--1560 (English).

\bibitem{Albu-Nastasescu1}
\bysame, \emph{Local cohomology and torsion theory. {I}}, Rev. Roumaine Math.
  Pures Appl. \textbf{26} (1981), no.~1, 3--14. \MR{616016 (84b:13012)}

\bibitem{stab}
John~A. Beachy, \emph{Stable torsion radicals over {FBN} rings}, J. Pure Appl.
  Algebra \textbf{24} (1982), no.~3, 235--244. \MR{656845 (84b:16001)}

\bibitem{BN}
Marcel B{\"o}kstedt and Amnon Neeman, \emph{Homotopy limits in triangulated
  categories}, Compositio Math. \textbf{86} (1993), no.~2, 209--234.
  \MR{1214458 (94f:18008)}

\bibitem{CPS2}
Wojciech Chach\'olski, Amnon Neeman, Wolfgang Pitsch, and J\'er\^ome Scherer,
  \emph{Relative injective resolutions via truncations, part {II}}.

\bibitem{CPS1}
Wojciech Chach\'olski, Wolfgang Pitsch, and J\'er\^ome Scherer, \emph{Relative
  injective resolutions via truncations, part {I}}.

\bibitem{CPS}
\bysame, \emph{Injective classes of modules}, Journal of Algebra and Its
  Applications \textbf{12} (2013), no.~4, 1250188, 13. \MR{3037263}

\bibitem{PS}
Wojciech Chach{\'o}lski and J\'er\^ome Scherer, \emph{Homotopy theory of
  diagrams}, Mem. Amer. Math. Soc. \textbf{155} (2002), no.~736, x+90.
  \MR{1879153 (2002k:55026)}

\bibitem{Dauns}
John {Dauns} and Yiqiang {Zhou}, \emph{{Classes of modules}}, Boca Raton, FL:
  Chapman \& Hall/CRC, 2006 (English).

\bibitem{Dwyer}
William~G. Dwyer and Jan Spali{\'n}ski, \emph{Homotopy theories and model
  categories}, Handbook of algebraic topology, North-Holland, Amsterdam, 1995,
  pp.~73--126. \MR{1361887 (96h:55014)}

\bibitem{Gabriel}
Pierre Gabriel, \emph{Des cat\'egories ab\'eliennes}, Bull. Soc. Math. France
  \textbf{90} (1962), 323--448. \MR{0232821 (38 \#1144)}

\bibitem{Gil07}
James Gillespie, \emph{Kaplansky classes and derived categories}, Math. Z.
  \textbf{257} (2007), no.~4, 811--843. \MR{2342555 (2009e:55030)}

\bibitem{G}
Jonathan~S. Golan, \emph{Localization of noncommutative rings}, Marcel Dekker
  Inc., New York, 1975, Pure and Applied Mathematics, No. 30. \MR{0366961 (51
  \#3207)}

\bibitem{Golan-libro}
\bysame, \emph{{Torsion theories}}, Pitman Monographs and Surveys in Pure and
  Applied Mathematics, vol.~29, {Longman Scientific \& Technical}, Harlow, 1986
  (English). \MR{880019 (88c:16034)}

\bibitem{Golan-Raynaud}
Jonathan~S. Golan and Jacques Raynaud, \emph{Derived functors of the torsion
  functor and local cohomology of noncommutative rings}, J. Austral. Math. Soc.
  Ser. A \textbf{35} (1983), no.~2, 162--177. \MR{704422 (84g:16025)}

\bibitem{GR}
Robert Gordon and J.~Chris Robson, \emph{The {G}abriel dimension of a module},
  J. Algebra \textbf{29} (1974), 459--473. \MR{0369425 (51 \#5658)}

\bibitem{Grothendieck}
Alexander Grothendieck, \emph{Sur quelques points d'alg\`ebre homologique},
  T\^ohoku Math. J. (2) \textbf{9} (1957), 119--221. \MR{0102537 (21 \#1328)}

\bibitem{HX}
Amit Hogadi and Chenyang Xu, \emph{Products, homotopy limits and applications},
  \url{http://arxiv.org/abs/0902.4016} (2009), (13 pages).

\bibitem{hovey}
Mark Hovey, \emph{Cotorsion pairs, model category structures, and
  representation theory}, Math. Z. \textbf{241} (2002), no.~3, 553--592.
  \MR{1938704 (2003m:55027)}

\bibitem{holoc}
Henning Krause, \emph{A note on cohomological localization}, arXiv:0611101,
  2006.

\bibitem{McL}
Saunders Mac~Lane, \emph{Categories for the working mathematician}, second ed.,
  Graduate Texts in Mathematics, vol.~5, Springer-Verlag, New York, 1998.
  \MR{1712872 (2001j:18001)}

\bibitem{matlis}
Eben Matlis, \emph{Injective modules over {N}oetherian rings}, Pacific J. Math.
  \textbf{8} (1958), 511--528. \MR{0099360 (20 \#5800)}

\bibitem{papp}
Zolt{\'a}n Papp, \emph{On stable noetherian rings}, Trans. Amer. Math. Soc.
  \textbf{213} (1975), 107--114. \MR{0393120 (52 \#13930)}

\bibitem{Quillen}
Daniel~G. Quillen, \emph{Homotopical algebra}, Lecture Notes in Mathematics,
  No. 43, Springer-Verlag, Berlin, 1967. \MR{0223432 (36 \#6480)}

\bibitem{Localization}
Andrew Ranicki (ed.), \emph{Non-commutative localization in algebra and
  topology}, London Mathematical Society Lecture Note Series, vol. 330,
  Cambridge, Cambridge University Press, 2006. \MR{2222649 (2007f:55002)}

\bibitem{Ste}
Bo~Stenstr{\"o}m, \emph{Rings and modules of quotients}, Lecture Notes in
  Mathematics, Vol. 237, Springer-Verlag, Berlin, 1971. \MR{0325663 (48
  \#4010)}

\bibitem{Virili_tesi}
Simone Virili, \emph{Length functions of {G}rothendieck categories with
  applications to infinite group representations}, preprint (40 pages), June
  2013.

\end{thebibliography}

Address of the author:

\medskip
%
%
%
%
%
%
%
%
%
%
Simone Virili - {\tt virili@math.unipd.it}

Dipartimento di Matematica Pura ed Applicata, Universit\`a degli studi di Padova, Via Trieste 63,
35121 Padova, Italia.

\end{document}